\documentclass[a4paper,10pt]{article}

\usepackage[utf8x]{inputenc}
\usepackage[english]{babel}
\usepackage[T1]{fontenc}
\usepackage{amsmath,amscd}
\usepackage{amsfonts}
\usepackage{amssymb}
\usepackage{amsthm}
\usepackage{mathrsfs}
\usepackage[all,cmtip]{xy}
\usepackage{textcomp}
\usepackage{qtree}
\usepackage{url}
\usepackage{amsopn,latexsym,amscd}
\usepackage[pdftex]{graphicx}
\usepackage{color}
\usepackage{amsmath,graphicx}

\newtheorem{teo}{Theorem}[section]

\newtheorem{rem}{Remark}[section]
\newtheorem{prop}{Proposition}[section]
\newtheorem{cor}{Corollary}[section]
\newtheorem{es}{Example}[section]
\newtheorem{defin}{Definition}[section]

\newcommand{\Z}{{\mathbb{Z}}}
\newcommand{\C}{{\mathbb{C}}}
\newcommand{\R}{{\mathbb{R}}}
\newcommand{\Q}{{\mathbb{Q}}}
\newcommand{\N}{{\mathbb{N}}}

\newcommand{\G}{{\mathcal{G}}}
\newcommand{\B}{{\mathcal{B}}}
\newcommand{\BG}{B{\mathcal{G}}}
\newcommand{\DG}{D{\mathcal{G}}}
\newcommand{\MS}{{\mathcal{S}}}
\newcommand{\Proj}{{\mathbb{P}}}
\newcommand{\A}{{\mathcal{A}}}

\newcommand{\Sc}{{\mathcal{S}}}
\newcommand{\Cc}{{\mathcal{C}}}
\newcommand{\F}{{\mathcal{F}}}
\newcommand{\T}{{\mathcal{T}}}
\newcommand{\U}{{\mathcal{U}}}
\newcommand{\Hc}{{\mathcal{H}}}
\newcommand{\oG}{{\overline {\mathcal  G}}}

\newcommand{\supp}{{\text{supp}}}

\title{Families of  building sets and regular wonderful models }
\author{Giovanni Gaiffi, Matteo Serventi}
\date{\today}

\begin{document}

\maketitle

\begin{abstract}
Given a subspace arrangement, there are several De
Concini-Procesi
models associated to it,   depending on   
  distinct sets 
of initial combinatorial data (building sets).
The first goal of this paper is   to   describe, for the root arrangements of types  \(A_n\), \(B_n\) (=\(C_n\)), \(D_n\),  the poset of  all the  building sets  which are invariant with respect to the Weyl group action, and therefore  to  classify all the wonderful models which are obtained by adding to the complement of the arrangement  an equivariant   divisor. 
Then we  point out, for every fixed \(n\), a family of  models  which includes the minimal  model and the maximal model; we  call these  models {\em regular models} and we  compute, in the complex case,  their Poincar\'e polynomials.
\end{abstract}

\section{Introduction} 
In \cite{DCP2}, \cite{DCP1},   De Concini and Procesi constructed  {\em wonderful  models}  for the   complement of a
subspace arrangement in a  vector space.  These are smooth varieties, proper over the given space, in which the union of the subspaces is replaced by a divisor with normal crossings.

The  interest in these  varieties was at first motivated by an  approach to Drinfeld construction of special
solutions for Khniznik-Zamolodchikov equation (see \cite{drinfeld}).  Moreover, in \cite{DCP2} it was shown, using the cohomology description of these models  to give an explicit presentation of  a Morgan algebra, that the mixed Hodge structure and the rational homotopy type of the complement of a complex  subspace arrangement  depend only on the intersection lattice (viewed as a ranked poset).

 Then real and complex  De Concini-Procesi models
turned out to play a relevant role in several   fields of mathematical research:  subspace and toric  arrangements,  toric varieties and  tropical geometry, 
moduli spaces of curves,     configuration spaces, box splines, index theory,  discrete  geometry    (see for instance   \cite{DCP4}, \cite{DCP3}, 
\cite{etihenkamrai}, \cite{feichtner},    \cite{feichtnersturmfels},  \cite{postnikov},
\cite{postnikoreinewilli} and  \cite{zelevinski}).

In general, given a subspace arrangement, there are several De
Concini-Procesi
models associated to it,   depending on   
  distinct sets 
of initial combinatorial data ({\em building sets}, see Section \ref{subsecbuilding}). Among these building sets there are  always a minimal one and a maximal one
with respect to inclusion: as a consequence there are always a  minimal and a maximal De Concini-Procesi model. 

The importance of the minimal construction was immediately pointed out, but  real and complex  non minimal models (in particular maximal models) appeared in various contexts (see 
 \cite{callegaif},   \cite{DJS},  \cite{LTV},  \cite{SzenesVergne}).  For instance  it is well known that  the toric variety of type \(A_{n-1}\)    is isomorphic to the maximal model  associated to the boolean arrangement (see  \cite{HenPisa} for
further references).

In this paper we will deal with  the root arrangements of types \(A_n\), \(B_n\) (=\(C_n\)), \(D_n\).  
As our first goal   we will   describe, for  these arrangements,  the poset of  all the associated building sets (ordered by inclusion)  which are invariant with respect to the Weyl group action, and therefore we will classify all the wonderful models which are obtained by adding to the complement of the arrangement  an equivariant   divisor. 

Our second goal will be  to point out, for every fixed \(n\), a family of models (which we will call {\em regular models}), which includes the minimal  model and the maximal model, and to compute the Poincar\'e polynomials of all the models in this family.

To describe our results more in detail, let  us consider  for instance the \(A_{n-1}\) case:  we will introduce a partial order on the set \(\Lambda_n\) of all the partitions of \(n\), and we will define a family of \(S_n\) invariant  building sets  \(\G_\lambda\), where \(\lambda\in \Lambda_n\) is a {\em building partition}, i.e. it is \((n)\) or a partition with at least two parts greater than or equal to 2.

Then, given any subset \(\{\lambda^1,\lambda^2,...,\lambda^k\}\) of pairwise not comparable building partitions, we will show that  the union \(\{\G_{\lambda^1} \cup \G_{\lambda^2}\cup\cdots  \cup \G_{\lambda^k}\}\) is an \(S_n\) invariant  building set, and that all the  \(S_n\) invariant  building sets can be obtained in this way (see Theorem \ref{teoclassificazione}).

Some particularly regular objects come out of  this picture, i.e. the building sets \(\G_{s}(A_{n-1})\)  obtained as the union of the building  sets   \(\G_\lambda\) such that \(\lambda\) has exactly \(s\)  parts.
Therefore, for every \(n\geq 2\)  we have a family of  \(n-2\) {\em regular building sets}:
\[\G_{1}(A_{n-1})\subset  \G_{2}(A_{n-1}) \subset \cdots  \subset \G_{n-2}(A_{n-1})\]  where \(\G_{1}(A_{n-1})\) coincides with  the minimal building set and  \(\G_{n-2}(A_{n-1})\)  with  the maximal one.  
We will give formulas for the Poincar\'e  series (Section \ref{secinductivefamilies}) of all the {\em regular    models} \(Y_{ \G_{s}(A_{n-1})}\).  For \(s=1\)   this  series  is   the well known series for the moduli spaces of stable \(n+1\)-pointed curves of genus zero,  while in the case of maximal models the formulas we obtain are explicit sums and products of polynomials whose coefficients involve the Stirling numbers of the second kind  (different formulas for the  Poincar\'e polynomials of the maximal models were  described in \cite{GaiffiServenti}).  The formulas for   the intermediate models are ``interpolations'' between the formulas for the maximal and the minimal cases.  

We will also compute  formulas for the Poincar\'e series of some auxiliary  wonderful models of subspace arrangements  (see Theorem \ref{formulafamigliainduttivatilda}).

The classification of all the Weyl group equivariant models  in the  \(B_n\) case, and  the    computations of the Poincar\'e polynomials of the \(B_n\)  regular models, are provided  in  Sections \ref{buildingbn}  and \ref{secinductivefamiliesbn},  while the  \(D_n\)  case is studied in  Sections \ref{secdn} and \ref{regdn}.

Finally, we  will  point out   the  connection between our formulas  and  the rich combinatorics of the
corresponding real
 De Concini-Procesi models.  The real models can  be contructed, as it is well known,  by gluing nestohedra, and    from this  one obtains   formulas for their   Euler characteristics. 
Different  formulas for these  Euler characteristics can also be obtained  by evaluating  in \(q=-1\) the Poincar\'e polynomials of the corresponding complex models. 
From the comparison of these two different computations  one obtains  nice combinatorial equivalences (see Section \ref{seceuler}). \\

\noindent {\em Acknowledgements.} We  wish to thank Filippo Callegaro  and Andrea Maffei for their  useful suggestions.

\section{Basic construction}
\subsection{Building sets and nested sets}
\label{subsecbuilding}
Let $V$ be a finite dimensional vector space 
and let  $\G$ be a finite set of subspaces of the dual space $V^*$.  We denote  by $\Cc_\G$ its closure under the sum.

\begin{defin}
 Given a subspace $U\in\Cc_\G$, a \textbf{decomposition of} $U$ in $\mathbf{\Cc_\G}$ is a collection
$\{U_1,\cdots,U_k\}$ ($k>1$) of non zero subspaces in $\Cc_\G$ such that
\begin{enumerate}
 \item $U=U_1\oplus\cdots\oplus U_k$
 \item for every subspace $A\subset U$, $A\in\Cc_\G$, we have $A\cap U_1,\cdots,A\cap U_k \in \Cc_\G$ and
$A=\left(A\cap U_1\right)\oplus\cdots\oplus \left(A\cap U_k\right)$.
\end{enumerate}
\end{defin}
\begin{defin}
 A subspace $F\in\Cc_\G$ which does not admit a decomposition  is called \textbf{irreducible} and the set of
irreducible subspaces is denoted by $\mathbf{\F_\G}$.
\end{defin}
One can prove  that 
every subspace $U\in\Cc_\G$ has a unique decomposition into irreducible subspaces.

\begin{defin}\label{building}
 A collection $\G$ of subspaces of $V^*$ is called \textbf{building} if every element $C\in\Cc_\G$ is the direct sum
$G_1\oplus\cdots\oplus G_k$ of the set of maximal elements $G_1,\cdots,G_k$ of $\G$ contained in $C$.
\end{defin}

As first examples of building sets one can consider the  set of irreducible subspaces  of a given family of subspaces of $V^*$, or any    set  of subspaces of $V^*$ which is closed under the sum.

Given a family $\G$ of subspaces of $V^*$ there are different sets \({\mathcal B}\) of subspaces of $V^*$ such that
\(\Cc_{{\mathcal B}}=\Cc_\G\); if we order by inclusion the collection of such sets, it turns out  that the minimal
element is
$\F_\G$ and the maximal one is $\Cc_\G$.
\begin{defin}(see \cite{DCP1}) \label{Gnested}
 Let $\G$ be a building set of subspaces of $V^*$. A subset $\Sc\subset \G$ is called $\mathbf{\G}$\textbf{-nested} if and only if 
 for every  subset $\{A_1,\cdots,A_k\}$ (\(k\geq 2\)) of pairwise non comparable elements of $\Sc$  the subspace   $A=A_1+\cdots +
A_k$ does not belong to  $\G$.

\end{defin}

We notice that if  $\Cc$ is a building family of subspaces closed under the sum, then   the subspaces of 
a $\Cc$-nested set  are  totally ordered (with respect to inclusion).
For a more general definition of building sets and nested sets from a purely combinatorial viewpoint see \cite{feichtnerkozlovincidence}.

\subsection{Wonderful models}
\label{subsecwonderful}

Let us take  \(\C\)  as the base field and consider a finite subspace arrangement in the complex vector space $V$. We will describe this arrangement by the dual arrangement \(\G\) in \(V^*\) (for every $A\in\G$, we will  denote by  $A^\perp$ its
annihilator in $V$). The complement in \(V\) of the arrangement will be denoted by \(\A_\G\).\\
For every $A\in\G$ we have a rational map defined outside of \(A^{\perp}\):
$$\pi_A:V\longrightarrow V/A^\perp \longrightarrow \Proj \left( V/A^\perp \right).$$

We then consider the embedding 
$$\phi_\G:\A_\G\longrightarrow V\times\prod_{A\in\G}\Proj\left( V/A^\perp \right)$$
given by the  inclusion on the first component and by the maps \(\pi_A\) on the other components. \begin{defin}
The De Concini-Procesi model
\(Y_{\G}\) associated to  $\G$ is the closure of $\phi_\G \left(\A_\G \right)$ in
$V\times\prod_{A\in\G}\Proj\left(V/A^\perp\right)$.
\end{defin}

These  {\em wonderful models}  are  particularly interesting when the arrangement $\G$ is building: they turn out to be 
smooth varieties and the complement of \(\A_\G\) in \(Y_\G\) is a divisor with normal crossings. The irreducible components of this divisor are in correspondence with the elements of \(\G\), and their intersection are described by the following rule: let us consider a subset   \(S\) of \(\G\); then the common intersection of the irreducible components associated to the elements of \(S\)   is nonempty if and only if \(S\) is a 
\(\G\)-nested set.  

The  integer cohomology rings of the models \(Y_\G\) have been described in  \cite{DCP1}. They are torsion free, and in  \cite{YuzBasi}  Yuzvinski explicitly described  $\Z$-bases (see also  \cite{GaiffiBlowups}).  We  briefly  recall these results.

Let $\G$ be a building set of subspaces of $V^*$.  If $\Hc\subset
\G$ and $B\in\G$ is such that $A\subsetneq B$ for each $A\in\Hc$,  one  defines
$$d_{\Hc,B}:=\dim B - \dim \left(\sum_{A\in\Hc} A\right).$$
In the polynomial ring $\Z[c_A]_{A\in\G}$, we consider the ideal \(I\) generated by the polynomials  
$$P_{\Hc,B}:=\prod_{A\in\Hc}c_A\left(\sum_{C\supset B}c_C\right)^{d_{\Hc,B}}$$
 as $\Hc$ and $B$ vary.
\begin{teo}(see \cite{DCP1}).\\
There is a surjective ring homomorphism
$$\phi \: : \: \Z[c_A]_{A\in\G}\longrightarrow H^*(Y_\G,\Z)$$
whose  kernel is $I$ and such that  \(\phi(c_A)\in H^2(Y_\G,\Z)\).  
\end{teo}
\begin{defin}

 Let $\G$ be a building set of subspaces of $V^*$. A function
$$f:\G\longrightarrow \N$$
is $\mathbf{\G}$\textbf{-admissible} (or simply \textbf{admissible}) if $f=0$ or, if $f\neq 0$, $\supp(f)$ is
$\G$-nested and for all $A\in\supp(f)$ one has
$$f(A)< d_{\supp(f)_A,A}$$
where $\supp(f)_A:=\{C\in\supp(f):C\subsetneq A\}$.
\end{defin}
\begin{defin}
 A monomial $m_f=\prod_{A\in\G}c_A^{f(A)}\in\Z[c_A]_{A\in\G}$ is \textbf{admissible} if $f$ is admissible.
\end{defin}
\begin{teo}\label{base coomologia}(see \cite{YuzBasi}, \cite{GaiffiBlowups})\\
 The set $\mathcal{B}_\G$ of all admissible monomials gives   a $\Z$-basis of $H^*(Y_\G,\Z)$.
\end{teo}

\section{A partial ordering on partitions}
Let us denote by  $\F_{A_{n-1}}$ the building set of irreducibles associated to the root system \(A_{n-1}\). There is a bijective correspondence between the elements of $\F_{A_{n-1}}$ and the subsets of $\{1,\cdots,n\}$ of cardinality at
least two:  if the annihilator of \(A\in \F_{A_{n-1}}\) is the subspace  described by the equation  \(x_{i_1}=x_{i_2}= \cdots =x_{i_k}\) then we represent \(A\) by the set \(\{i_1,i_2,\ldots, i_k\}\).
In an analogous way we can establish a bijective correspondence between the   elements of the maximal building set ${\mathcal C}_{A_{n-1}}$  and  the unorderd partitions of the set $\{1,\cdots,n\}$ in which at least one part has more than one element:
for instance, \(\{1,3,4\}\{2,5\}\{6\}\{7,8\}\) represents the  subspace  in ${\mathcal C}_{A_{7}}$ of  dimension 4  whose annihilator is described by the system of equations \(x_{1}=x_{3}= x_{4}\), \(x_{2}=x_{5}\) and \(x_{7}=x_{8}\).

Let us denote by \(\Lambda_n\) the set of partitions of \(n\in \N\). 
To every unordered partition of $\{1,\cdots,n\}$ we can associate, considering the cardinalities of its parts, a partition in \(\Lambda_n\). Therefore we can associate a partition in \(\Lambda_n\) to every subspace in ${\mathcal C}_{A_{n-1}}$. We will say that a subspace in ${\mathcal C}_{A_{n-1}}$  has the form \(\lambda\in \Lambda_n \) if its associated partition is \(\lambda\). 
For instance, the subspace \(\{1,3,4\}\{2,5\}\{6\}\{7,8\}\)  in ${\mathcal C}_{A_{7}}$ has the form \((3,2,2,1)\).

In this section we will describe  a poset structure on \(\Lambda_n\) which will be used in  the classification of all the \(S_n\) invariant building sets associated   to the root system \(A_{n-1}\).

If \(n\geq 1\) and \(\lambda \in \Lambda_n\), we will  represent \(\lambda\)  by its Young diagram and call {\em admissible} the following moves: 
\begin{itemize}
\item[a)]  remove  an {\em entire} row and add all its boxes to   another  row which has at least two boxes;  then, if necessary, rearrange the rows in order to obtain a Young diagram (see Figure \ref{mossa});\\
\item[b)]   remove \(k\geq 2\) rows made by a single box and form a row made by \(k\) boxes, if \(k\) is greater than or equal to the number of boxes of  the smallest  row with more than one box;  then, if necessary, rearrange the rows in order to obtain a Young diagram (see Figure \ref{mossabis}). 
\end{itemize}

 \begin{figure}[h]
 \center
\includegraphics[scale=0.25]{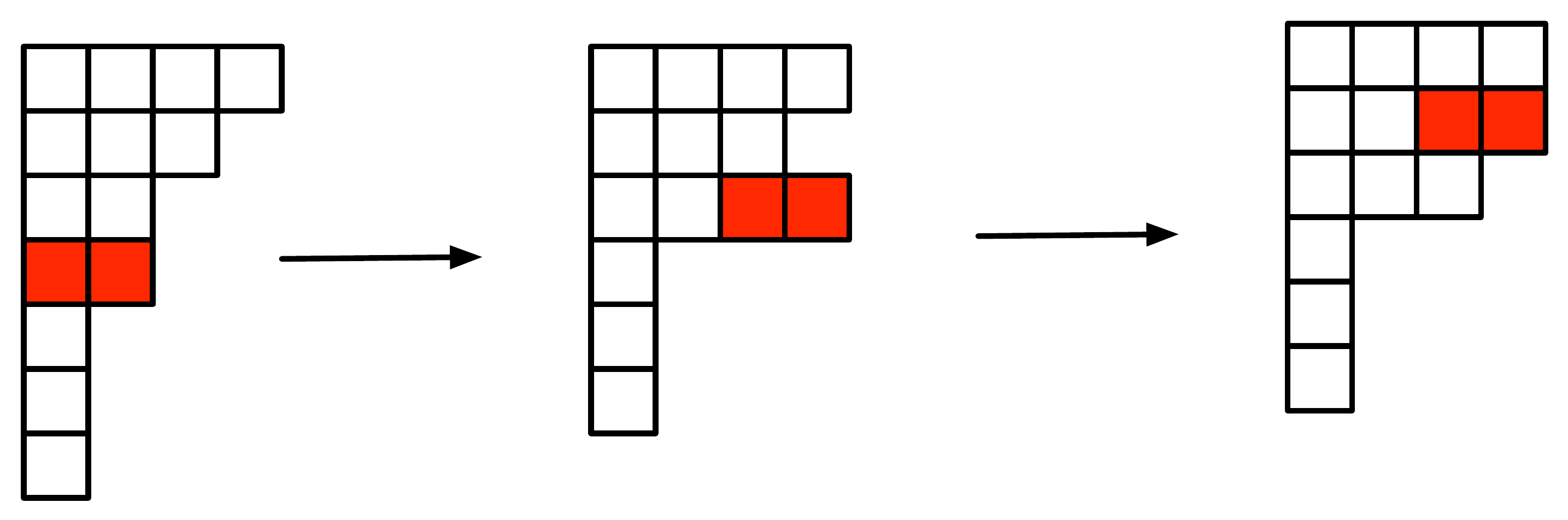}
\caption{Example of an admissible move of type \(a)\).}
\label{mossa}

\end{figure}

 \begin{figure}[h]
 \center
\includegraphics[scale=0.25]{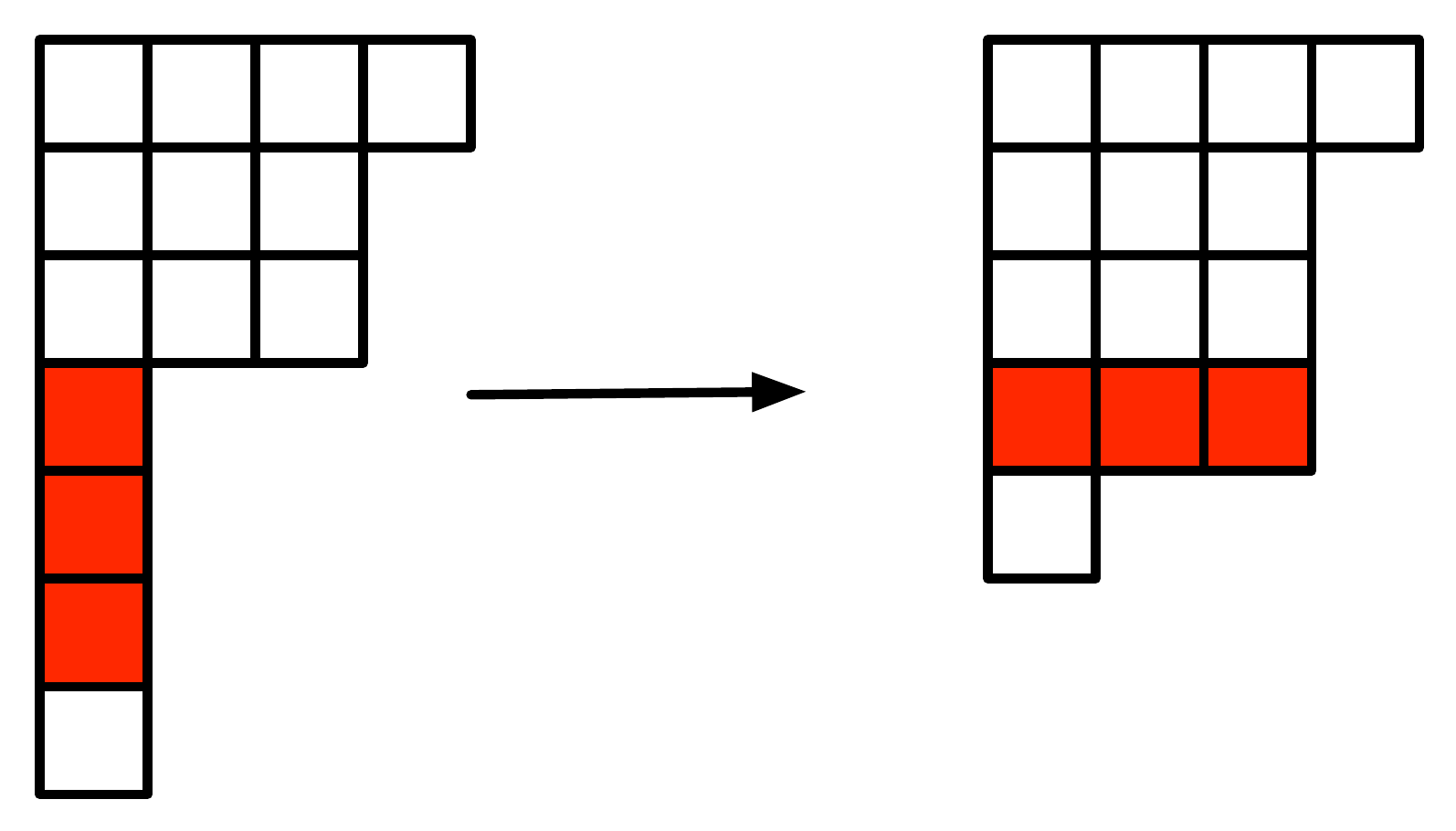}

\caption{Example of an admissible move of type \(b)\). The number of  boxes of the new row is  \( 3\): if it was 2, the move would not be admissible.}
\label{mossabis}
\end{figure}
\begin{rem}
If \(\lambda=(1,1,...,1)\) there are no possible admissible moves.

\end{rem}
Now we equip \(\Lambda_n\) with the following partial order:  \(\lambda \in \Lambda_n\) is greater than  \(\mu\in \Lambda_n\) (we write \(\lambda>\mu\))  if \(\lambda\neq \mu\) and the Young diagram of \(\lambda \) can be obtained by the one of \(\mu\) by a sequence of admissible moves (see Figure \ref{mossadue}).

 \begin{figure}[h]
 \center
\includegraphics[scale=0.25]{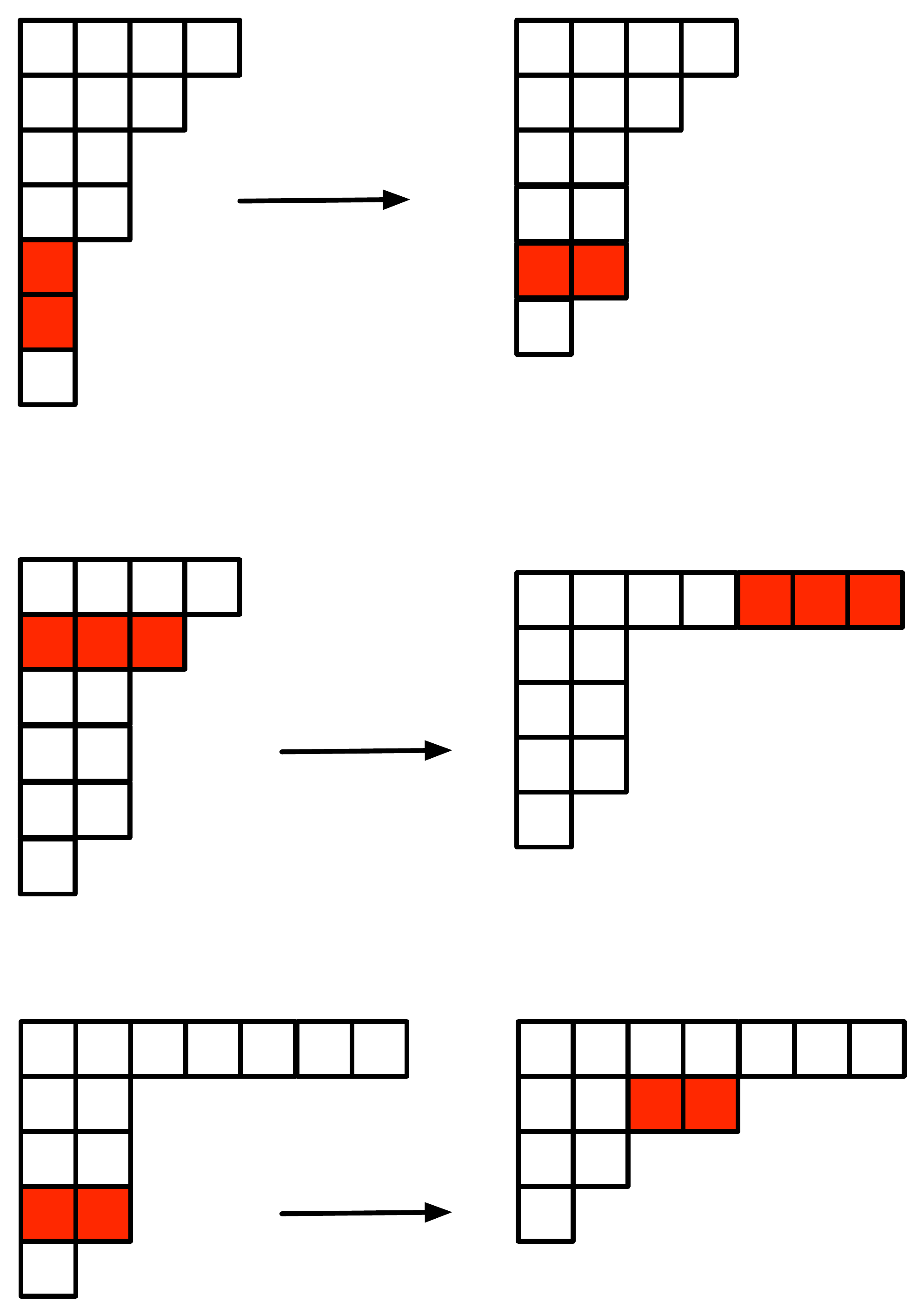}

\caption{A sequence of three admissible moves, starting from the partition \(\mu=(4,3,2,2,1,1,1)\). At the end we obtain \(\lambda=(7,4,2,1)\), therefore \(\lambda>\mu\).}
\label{mossadue}
\end{figure}

In the sequel we will be interested in the subset \(B\Lambda_n\) of \(\Lambda_n\) (\(n\geq 1\))  made by \((n)\) and, if \(n\geq 4\),  by all the partitions   with at least two numbers greater than or equal to 2, i.e.
\(\lambda \in B\Lambda_n\) iff \(\lambda=(n)\) or 
\(\lambda=(\lambda_1,\lambda_2,\ldots,\lambda_k,...)\)   with \(k\geq 2\) and 
\[\lambda_1\geq \lambda_2 \geq \cdots  \geq \lambda_k\geq 2\]
We will call {\em building partitions} the partitions in \(B\Lambda_n\).
The ordering of \(\Lambda_n\) induces a poset structure on \(B\Lambda_n\).

\begin{rem}
\label{remconfronto}
Let \(\gamma, \delta\) be two  building partitions.
If \(\gamma\geq \delta\) one can find a subspace of the form \(\gamma\)
 which contains a subspace of the form \(\delta\).
\end{rem}
\begin{rem}
Let us denote by \(\succeq\)  the well known partial ordering on \(\Lambda_n\) such that 
 \(\mu\succeq \gamma\) if and only if \(\mu_1\geq \lambda_1\) and \(\mu_1+\mu_2 \geq \lambda_1+\lambda_2\) and so on.
 We observe that \(\gamma\geq \lambda\)  implies \(\gamma\succeq \lambda\) but the reverse implication is not true. In fact the ordering \(\succeq\) can be obtained as a result of a set of moves  which includes the moves used to define  \(\geq\): the elementary steps consist in removing  a box from a row of a Young diagram and adding  it  to a higher row.
We notice that, 
for instance, 
\((5,2)\succeq (4,3)\) but  one  cannot find a subspace of the form \((4,3)\) inside a subspace of the form \((5,2)\).

\end{rem}

\begin{rem}
\label{notvee}
Given two partitions  \(\lambda, \gamma\)  in the poset \((B\Lambda_n, \geq )\), it is not true that there exists a minimum element \(\mu\in B\Lambda_n\) such that \(\mu\geq \lambda\) and \(\mu \geq \gamma\). Let us consider for instance \(\lambda=(8,4,4), \gamma=(7,5,3,1)\). The (not comparable) partitions \(\theta= (12,4)\) and \(\rho=(8,8)\) are the minimal partitions in \(B\Lambda_n\) which are  \(\geq \lambda, \gamma\).
Furthermore, it is not true that  there exists a maximum element \(\mu\in B\Lambda_n\) such that \( \theta\geq \mu\) and \(\rho \geq \mu\).
\end{rem}

\section{The \(S_n\) invariant   building sets of type \(A_{n-1}\)}
We are going to  to describe all the  building sets associated to the root arrangement \(A_{n-1}\) which are invariant  with respect to  the natural \(S_n\) action (this is in the spirit of the construction of the compactifications of configuration spaces:  the corresponding wonderful models will have a \(S_n\) equivariant divisor at the boundary).

We start by defining a family of building sets, parametrized   by building partitions.
\begin{defin} Let \(\lambda\) be a building partition. We define \({\widetilde \G}_{\lambda}\) as  the set made by  all the subspaces of the form 
 \(\gamma\in \Lambda_n\) for every \(\gamma\geq\lambda\).  We define \(\G_{\lambda}\) as  \({\mathcal F}_{A_{n-1}}\cup {\widetilde \G}_{\lambda} \). 

\end{defin}

\begin{rem} We notice that, according to the definition, if  \(\lambda=(n)\) then \(\G_{\lambda}\) is the building set of irreducibles \({\mathcal F}_{A_{n-1}}\). The  only     building set  associated to the root system \(A_2\) (i.e. when \(n=3\)) is \(\G_{(3)}={\mathcal F}_{A_{2}}\).   If \(n= 4\), there are two building sets:  the minimal one \(\G_{(4)}={\mathcal F}_{A_{3}}\) and the maximal  one  \(\G_{(2,2)}\).  If \(n> 4\), the maximal building set is \(\G_{(2,2,1,1,1,...)}\).
\end{rem}

It is immediate from the definition that:

\begin{prop}
Given two different building partitions \(\mu\) and  \(\lambda\), the   building set \(\G_\lambda\) is  included in \(\G_\mu\) if and only if \(\lambda>\mu\).
\end{prop}

 

 \begin{figure}[h!]
 \center
\includegraphics[scale=0.25]{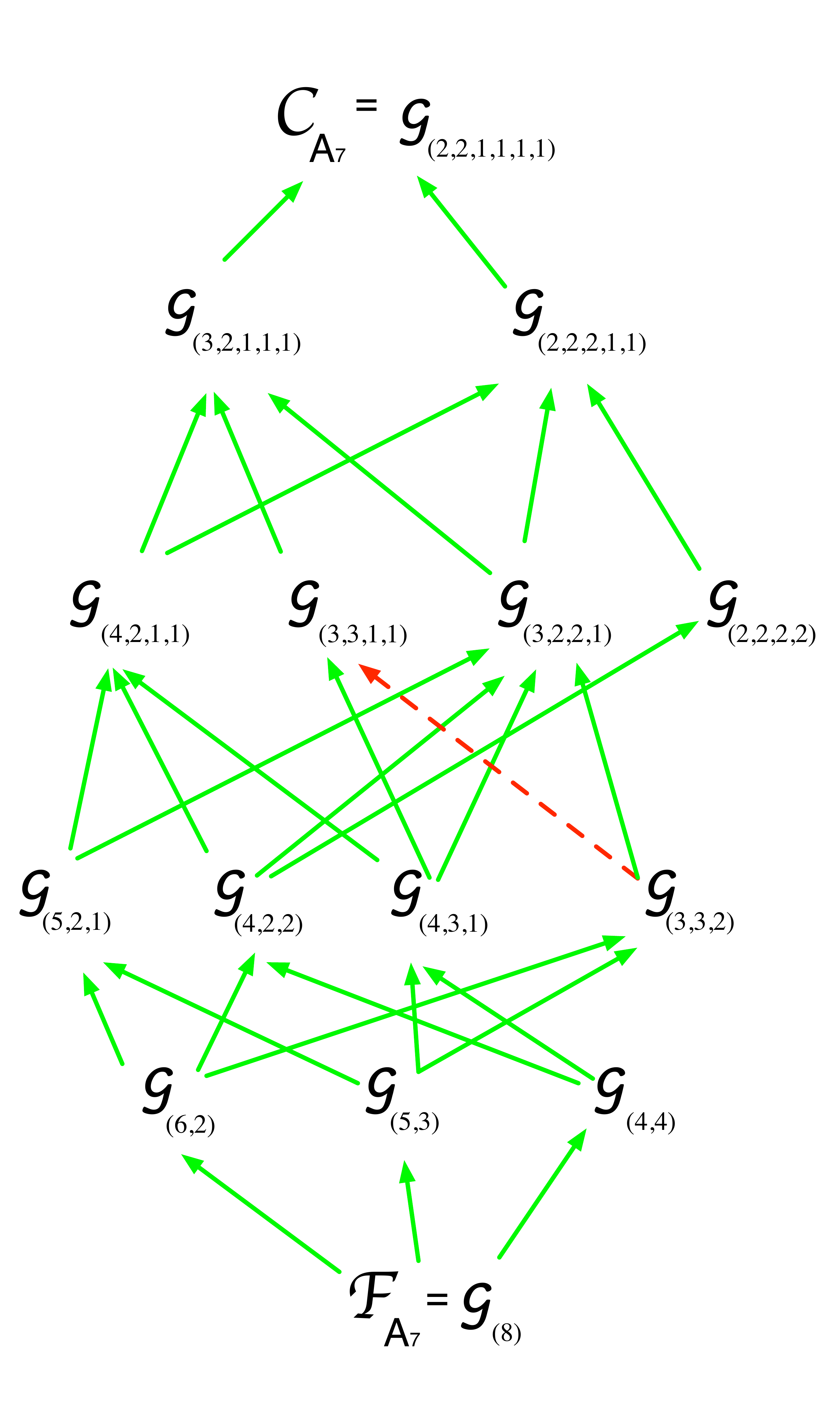}

\caption{The Hasse diagram of the family  of building sets of type \(\G_\lambda\) in the case \(A_7\) (green arrows represent   inclusions).  
\label{famigliabuilding3}
The red  arrow shows a case where there is not inclusion.}
\end{figure}

The   building sets of type \(\G_\lambda\) (\(\lambda\in B\Lambda_n\)) are not the only \(S_n\) invariant building sets which include \({\mathcal F}_{A_{n-1}}\). For instance, in the \(A_5\) case, \(\G_{(4,2)}\cup \G_{(3,3)}\) is an \(S_6\) invariant building set, which does not belong  to  the family \(\G_\lambda\). The following theorem describes all the \(S_n\) invariant building sets.
\begin{defin}
Given a set \(\MS=\{\lambda^1,\lambda^2,...,\lambda^v\} \) of pairwise not comparable elements in \(B\Lambda_n\), we denote by \(\G_{\MS}\) the building set 
\[\G_{\MS}=\G_{\lambda^1}\cup \G_{\lambda^2}\cup...\cup\G_{\lambda^v}\]
We denote by \({\mathcal T}_n\) the set whose elements are the   nonempty sets of pairwise not comparable elements  in \(B\Lambda_n\).
\end{defin}

\begin{teo}
\label{teoclassificazione}
The map:
\[\MS=\{\lambda^1,\lambda^2,...,\lambda^v\}\rightarrow \G_{\MS}\]
is a bijection between \({\mathcal T}_n\)  and the set made by  the \(S_n\) invariant building sets which contain \({\mathcal F}_{A_{n-1}}\).
\end{teo}
\begin{proof}
Let us consider a \(S_n\) invariant building set \(\B\) which contains \({\mathcal F}_{A_{n-1}}\).  If it is different from  \({\mathcal F}_{A_{n-1}}\) we consider the subspaces in \(\B-{\mathcal F}_{A_{n-1}}\): to each of these subspaces  we associate the partition in \(B\Lambda_n\) which describes its form and, among these partitions, we choose the minimal ones (with respet to \(\geq\)).

Let \(A\) be   a subspace in \(\B-{\mathcal F}_{A_{n-1}}\) whose associated partition \(\gamma= (\gamma_1,\gamma_2,\ldots , \gamma_s)\in B\Lambda_n\) is minimal.   Then, by \(S_n\) invariance, \(\B\) contains all the subspaces of this form. We will show that \(\G_{\gamma}\subseteq \B\). For this it suffices to show that if \(\mu \in B\Lambda_n\) can be obtained from \(\gamma\) by an admissible move, than \(\B\) contains all the subspaces of the form \(\mu\).
Let us consider moves of type \(a)\):  then   \(\mu=(\mu_1,\mu_2,\ldots, \mu_i,..., \mu_{s-1})\) where the numbers \(\mu_l\) coincide with the \(\gamma_h\) except for \(\mu_i=\gamma_j+\gamma_t\), where \(\gamma_j>1\).  Now we take   two subspaces which are of the form \(\gamma\) and therefore belong to  \(\B\):
\[C=\underset{\gamma_1}{ \{\ldots  \}}\underset{\gamma_2}{ \{ \ldots  \}}...\underset{\gamma_j}{ \{1,2,3,...,\gamma_j\}}\cdots \underset{\gamma_t}{ \{\gamma_j+1,\gamma_j+2,...,\gamma_j+\gamma_t\}}\cdots \underset{\gamma_s}{ \{ \ldots  \}}\]
\[D=\underset{\gamma_1}{ \{\ldots  \}}\underset{\gamma_2}{ \{ \ldots  \}}...\underset{\gamma_j}{ \{\gamma_j+1,2,3,...,\gamma_j\}}\cdots \underset{\gamma_t}{ \{1,\gamma_j+2,...,\gamma_j+\gamma_t\}}\cdots \underset{\gamma_s}{ \{ \ldots  \}}\]
The sum \(C+D\) is the subspace
\[C+D=\underset{\gamma_1}{ \{\ldots  \}}\underset{\gamma_2}{ \{ \ldots  \}}...\underset{\gamma_j+\gamma_t}{ \{1,2,3,...,\gamma_j,\gamma_j+1,\gamma_j+2,...,\gamma_j+\gamma_t\}}\cdots \underset{\gamma_s}{ \{ \ldots  \}}\]
By definition of building set, \(C+D\) must be the direct sum of the maximal subspaces in  \(\B\) contained in it. This is possible only if \(C+D \in \B\).
We have shown that \(\B\) contains a subspace of the form \(\mu\), and therefore it contains all such subspaces.

As for the moves  of type \(b)\), let \(\gamma_k\) be the last part  which is \(>1\) of \(\gamma \in B\Lambda_n\).  As a particular case, we first show that if \(\mu\) is obtained from \(\gamma\) by deleting \(\gamma_k\) parts equal to \(1\) and adding a part equal to \(\gamma_k\) then  \(\B\) contains a subspace of the form \(\mu\). The argument is similar to the one above. For instance, if \(\gamma=(4,3,1,1,1,1)\) and \(\mu=(4,3,3,1)\),  one then considers the two subspaces:
\[C=\{1,2,3,4\}\{5,6,7\}\{8\}\{9\}\{10\}\{11\}\]
\[D=\{1,2,3,4\}\{8,9,10\}\{5\}\{6\}\{7\}\{11\}\]
The sum \(C+D\) is the subspace
\[C+D=\{1,2,3, 4\}\{5,6,7\}\{8,9,10\}\{11\}\]
which must belong to \(\B\) and has the form \(\mu\).
Combining the result in this particular case with  the result for moves of the first type,   it is now easy to prove that  if \(\mu \in B\Lambda_n\) can be obtained from \(\gamma\) by any  admissible move of the second type, than \(\B\) contains all the subspaces of the form \(\mu\).

Let \(S=\{\gamma^1,\gamma^2, ...,\gamma^v\}\) be  the set of the (pairwise not comparable)  minimal partitions associated to the subspaces in \(\B-{\mathcal F}_{A_{n-1}}\). Repeating the  argument described  above  we can prove that \(\B\) contains \(\G_{\gamma^1}\cup \G_{\gamma^2}\cup...\cup\G_{\gamma^v}\).

To show the reverse inclusion,  let us consider  \(D\in \B-{\mathcal F}_{A_{n-1}}\). If \(D\) is associated to a minimal partition, say \(\gamma^1\),  then  \(D\in \G_{\gamma^1}\) by definition.
If the partition \(\gamma\) associated to \(D\) is not minimal, then for a certain \(i\) we have \(\gamma>\gamma^i\).  By  definition of \(\G_{\gamma^i}\) we know that \(D\in \G_{\gamma^i}\):  this concludes  the proof that \(\B=\G_{\gamma^1}\cup \G_{\gamma^2}\cup...\cup\G_{\gamma^v}\).

Now we must show that the above expression for \(\B\) is unique, i.e., if \(\B=\G_{\lambda^1}\cup \G_{\lambda^2}\cup...\cup\G_{\lambda^v}\) and \(\B=\G_{\theta^1}\cup \G_{\theta^2}\cup...\cup\G_{\theta^r}\) then \(r=v\) and, up to reordering, \(\theta^i=\lambda^i\) \(\forall i\).
Let us suppose that \(\B\neq {\mathcal F}_{A_{n-1}}\) (otherwise the statement is trivial).

First we observe that if \(\lambda^1\) is  not \(\geq\) of  one of the partitions  \(\theta^1,\theta^2,...,\theta^r\), then in  \(\G_{\theta^1}\cup \G_{\theta^2}\cup...\cup\G_{\theta^r}\) there are not elements of the form \(\lambda^1\). This is a contradiction. Therefore we must have, say,  \(\lambda^1\geq \theta^1\). The same argument shows that there exists \(i\) such that \(\theta^1\geq \lambda^i\). This implies \(\lambda^1\geq \lambda^i\), and since the elements \(\lambda^1,\lambda^2,...,\lambda^v\) are pairwise not comparable, we must have \(i=1\) and \(\lambda^1\geq \theta^1\geq\lambda^1\), that is to say, \(\lambda^1=\theta^1\).
The claim follows by induction.

\end{proof}

\section{The Poincar\'e polynomial for the maximal model (case \(A_{n-1}\))}
\label{secformulamax}
In this section we  provide a formula for  the Poincar\'e polynomial  of the maximal model \(Y_{\G_{(2,2,1,1,....)}}=Y_{max,n}\):
\[P(Y_{max,n})(q)=\sum_i dim \ H^{2i}(Y_{max,n}, \Q) q^i\]
We   use  a combinatorial strategy, different from the one in  \cite{GaiffiServenti},   which in the next sections will be generalized  in many ways  (i.e. it will be applied to different models and to different root arrangements).

  Let $G$ be a minimal (with respect to inclusion) 
element in a building set $\G$. Let $\G'=\G-\{G\}$, and 
	let $\oG$  be the  family in ${\displaystyle \left(\C^n 
	\right)^{\ast} /G}$ 
	given by the elements $\{ (A+G)/G\: : \: A\in \G'\}$. 
	In \cite{DCP1} and  \cite{Gaiffitesi}  it is shown  that $\G'$ and 
	$\oG$ are building and that $Y_{\G}$ can be obtained  by blowing up $Y_{\G'}$  along a subvariety isomorphic to $Y_{\oG}$.
	
 This implies that, denoting  by $p$ the blowing up map ${\displaystyle p\: : 
 \: Y_{\G} \mapsto Y_{\G'} }$, we have 
 \begin{displaymath}
 	H^{\ast}(Y_{\G}, \Z)\cong p^{\ast}H^{\ast}(Y_{\G'}, \Z) \oplus 
 	\left( H^{\ast}(E, \Z)/p^{\ast}H^{\ast}(Y_{\oG}, \Z) \right )
 \end{displaymath}
 The exceptional divisor $E$  is isomorphic to the 
 projectivization of the  normal bundle of $Y_{\oG}$ in $Y_{\G'}$.
 Then $H^{\ast}(E,\Z)$ is 
 generated, as $p^{\ast}H^{\ast}(Y_{\oG}, \Z)$-algebra, by the Chern 
 class $\zeta=c_1(T)$ of the tautological line bundle  $T\mapsto E$.
 Furthermore the class $\zeta$ has in $H^{\ast}(E,\Z)$ the unique relation
 provided by the Chern polynomial of the normal bundle 
 $N_{Y_{\oG}/Y_{\G'}}$.  This proves the following proposition where we  denote by 
 $P (Y)(q)  $  (\(q\) has degree 2) the Poincar\`e polynomial  of a model \(Y\): 
 
\begin{prop}
\label{propindu}
 Let $G$ be a minimal (with respect to inclusion) 
element in a building set $\G$. Then: 
 \begin{displaymath}
 	P (Y_{\G})(q)=P(Y_{\G'})(q)+\frac{q^{dim\, G }-q}{q-1}P(Y_{\oG})(q)
 \end{displaymath}
 \end{prop}

 \begin{teo}
 \label{formulamax}
 For \(n\geq 2\), we have the following inductive formula for the Poincar\'e polynomial of the maximal model \(Y_{\G_{(2,2,1,1,....)}}=Y_{max,n}\):
 \[P(Y_{max,n})(q)=1+\sum_{
\begin{array}{c}
\lambda \in \Lambda_n  \\
\lambda\neq (1,1,1,...,1) 
\end{array}
}\frac{q^{n-l(\lambda)}-q}{q-1} t_{\lambda} P(Y_{max,l(\lambda)})(q)\]
  where \(l(\lambda)\) is the {\em lenght} of the partition  \(\lambda\) (the number of parts) and \(t_\lambda\) is the number of subspaces whose  form is  \(\lambda\).
 \end{teo}
\begin{proof}
We obtain this formula   by  applying  Proposition \ref{propindu} several times. 
We start by choosing  a  minimal subspace in the building set \(\G_{(2,2,1,1,....)}\), then a minimal subspace in the building set \(\G_{(2,2,1,1,....)}'\) and so on.

The key observation is that  the `quotient' building sets which are produced by this process are all isomorphic to maximal building sets.
More precisely, let us suppose that, at a certain step, we have the building set \(\G_{(2,2,1,1,....)}''=\G_{(2,2,1,1,....)}-\{ \hbox{some subspaces} \}\). The  dimension of the deleted subspaces can be  bounded, since   at every step we have to remove a minimal subspace, so at first we can remove the subspaces of dimension 2, then the subspaces of dimension 3, and so on. Let us therefore suppose that in \(\G_{(2,2,1,1,....)}''\) the deleted subspaces are of dimension \(\leq j\). Now we remove a minimal subspace \(A\) from \(\G_{(2,2,1,1,....)}''\): if there still are subspaces of dimension \(j\) in \(\G_{(2,2,1,1,....)}''\) then \(A\) has dimension \(j\), otherwise \(A\) has dimension   \(j+1\).  Let \(A\) have   the form \(\lambda\). Now let us consider \(\overline   {\G''}\): it is isomorphic to a  building set associated to  the arrangement \(A_{l(\lambda)-1}\). In fact we can think of  the quotient space as the space where some groups of variables are equal: there are only \(l(\lambda)\) `free variables'. From this point of view,  it is easy to check that  \(\overline   {\G''}\) is   the maximal building set of type \(A_{l(\lambda)-1}\):  every subspace in this maximal building set can be obtained  as a quotient \(B+A/A\), where we can choose \(B\in \G_{(2,2,1,1,....)}\) with  \(dim \ B > dim \ A\), so \(B\in \G_{(2,2,1,1,....)}''\). 
\end{proof}

 \begin{rem}
We put \(P(Y_{max,1})(q)=1\) as a base for the induction. Then we observe that \(P(Y_{max,2})(q)=1\) and \( P(Y_{max,3})(q)=q+1\).
\end{rem}

From this  inductive formula we can write  \(P(Y_{max,n})(q)\) as an explicit sum of polynomials whose coefficients are expressed in terms of the Stirling numbers of the second kind.
\begin{defin}
Given two positive integers \(n>j\), let us denote by \(f_{n,j}(q)\) the polynomial \(S(n,j)\frac{q^{n-j}-q}{q-1} \; \),    where \(S(n,j)=\frac{Surg(n,j)}{j!}\) is the Stirling number of the second kind. 

\end{defin}
\begin{teo}
\label{ulterioremiglioramentoformulamax}

 For \(n\geq 3\),   we have:
 \[P(Y_{max,n})(q)=1+\sum_{\begin{array}{c}
 1\leq k\leq \lfloor \frac{n-1}{2} \rfloor\\
(j_1,j_2,...,j_k)\in J_k(n-2)
\end{array}}f_{j_2,j_1}(q)f_{j_3,j_2}(q)\cdots f_{j_k,j_{k-1}}(q)f_{n,j_k}(q) \]
where \(J_k(n-2)\) is the set of all the lists \((j_1,j_2,...,j_k)\) of integers such that \(1\leq j_1<j_2<\cdots <j_k\leq n-2\) and, for every \(i=1,2,...,k\), \(j_i-j_{i-1}\geq 2\).

\end{teo}
\begin{proof}
The proof is by  induction on \(n\).

 One first   observes that, in the sum which appears in the formula of Theorem \ref{formulamax}, we can regroup all the partitions with the same length \(j\), with \(j=1,...,n-2\): 
 \[P(Y_{max,n})(q)=1+\sum_{
j=1,....,n-2
}\frac{q^{n-j}-q}{q-1} \; S(n,j) P(Y_{max,j})(q)\]

The conclusion then  follows by induction.  
\end{proof}

\begin{rem}
We notice that we can use our formula  for the Poincar\'e polynomials of the maximal models  to obtain formulas  for \(P(Y_{\G_{(2,2,1,1,....)}''})(q)\), where \(\G_{(2,2,1,1,....)}''\) is any one of the building sets obtained as a result of the above described algorithm,  which starts from \(\G_{(2,2,1,1,....)}\) and deletes at each   step a minimal subspace. In fact, at each  step of the algorithm we have a relation like the following one:
\[   P(Y_{max,n})(q)= P(Y_{\G_{(2,2,1,1,....)}''})(q)+ \hbox{sums of polynomials} \  \frac{q^{n-j}-q}{q-1} P(Y_{max,j})(q)\]

\end{rem}

In principle it is possible to use arguments  similar to the one used in the proof of  Theorem \ref{formulamax}  to  find formulas for the Poincar\'e polynomials of the models \(Y_{\G_\lambda}\) (\(\lambda \in B\Lambda_n\)), 
but  when we quotient by a subspace it is not  always  true that the quotient building set is one of the invariant   ones described in the preceding sections, so  the computation may need further steps and may become  more complicated. 

\section{Regular   building sets}
\label{secinductivefamilies}

The following building sets appear as natural objects in our picture, since they are obtained as unions of the building sets \(\G_\lambda\) which lie on a  same row of the diagram of \(B\Lambda_n\) (see Figure \ref{famigliabuilding3}). 

\begin{defin}
For every \(n\geq 3\) and \(1\leq s\leq n-2\) we denote by \(\G_s(A_{n-1})\) the \(S_n\) invariant building set  which contains \({\mathcal F}_{A_{n-1}}\) and also all the subspaces of the maximal building set which have dimension \(\geq n-s\). We will call \(\G_s(A_{n-1})\) the {\em regular building set of degree \(s\)}.
\end{defin}
We notice that, for \(n\geq 3\), \(\G_1(A_{n-1})= {\mathcal F}_{A_{n-1}}=\G_{(n)}\) and, for \(n\geq 4\),  \(\G_{n-2}(A_{n-1})\) is equal to the maximal building set \(\G_{(2,2,1,...,1)}\).
Therefore, for every  \(n\geq 3\) we have pointed out \(n-2\) distinct regular building sets which include the irreducibles:
\[\G_{1}(A_{n-1})\subset  \G_{2}(A_{n-1}) \subset \cdots  \subset \G_{n-2}(A_{n-1}).\]  

The following definition points out the property needed to apply the argument of the proof of Theorem \ref{formulamax} to more general building sets. 

\begin{defin} 
\label{defgeneraleinductive}
Let us consider, for every \(n\geq 1\),  a   building set  \(\G(n)\) (associated to  a subspace  arrangement in \(\C^n\)). We will call the family \(\{\G(n)\}\) {\em inductive} if, when we take any subspace  \(G \in \G(n)\) of dimension \(j\) (with \(n-1\geq j  > 2\)), the building set \({\overline \G(n)}=\{ (A+G)/G\: : \: A\in \G''\}\) is (isomorphic to) \(\G(n-j)\), for every  \(\G''\)  obtained from \(\G\) by removing \(G\),  all the subspaces of dimension \(<j\) and any collection of subspaces of dimension \(j\).
\end{defin}
 A first remark is that   the family of maximal building sets  is   inductive.
 
 We observe that  \(\G_{n-3}(A_{n-1})\) 
 is an inductive family, while \(\G_{n-4}(A_{n-1})\) is inductive ``up to subspaces of dimension 1'', that is to say, the quotient building sets may differ from the expected ones, but   only in the subspaces of dimension 1. For \(s<n-4\)  the family \(\G_s(A_{n-1})\)  is not inductive, so the argument of the proof of Theorem \ref{formulamax}  cannot be applied. 
Anyway we will manage to compute the Poincar\'e polynomials of the associated models.  For this it is  useful to introduce  some different families of building sets, which are inductive:
\begin{defin}
For every \(n\geq 3\) and \(1\leq s\leq n-1\) we denote by \({\widetilde \G}_s(A_{n-1})\) the (\(S_n\) invariant)  building set  which contains all the subspaces of the maximal building set which have dimension \(\geq n-s\).
\end{defin}

\begin{rem}
We notice that \({\widetilde \G}_{n-1}(A_{n-1})=\G_{n-2}(A_{n-1})\) and that  all the other building sets \({\widetilde \G}_{s}(A_{n-1})\)  (when \(1\leq s \leq n-2\)) do not include hyperplanes.
All the families  \({\widetilde \G}_s(A_{n-1})\) are inductive.
\end{rem}

 For convenience of notation,     for every \(n\geq 2\) and \(s\geq  n-1\) we put  \(\G_s(A_{n-1})\) and \({\widetilde \G}_{s+1}(A_{n-1})\)  to be equal to the maximal building set.
 Let us denote by \(f_{n,j}(q)\), as in Section \ref{secformulamax},     the polynomial \(S(n,j)\frac{q^{n-j}-q}{q-1} \; \) (where \(S(n,j)\) are the Stirling numbers of the second kind).

 \begin{teo}
 \label{formulafamigliainduttivatilda}
For every \(n\geq 3\) and  \(1\leq s\leq n-2\)  we have the following  formula for the Poincar\'e polynomial of the models \(Y_{{\widetilde \G}_{s}(A_{n-1})}\):
 \[P(Y_{{\widetilde \G}_{s}(A_{n-1})})(q)=1+\sum_{
\begin{array}{c}
1\leq k \leq \lfloor \frac{s+1}{2} \rfloor\\
( j_1, j_2, ..., j_k)\in J_k(s) 
\end{array}
}f_{j_2,j_1}(q)f_{j_3,j_2}(q)\cdots f_{j_k,j_{k-1}}(q)f_{n,j_k}(q)\]
 \end{teo}
\begin{proof}
One repeats the steps of  the proofs of Theorems \ref{formulamax} and \ref{ulterioremiglioramentoformulamax}, paying attention to the fact that the length of the  partitions which appear is \(\leq  s\).
\end{proof}

\begin{rem}
Since  \({\widetilde \G}_{n-2}(A_{n-1})\) and \(\G_{n-2}(A_{n-1})\) differ only in the subspaces of dimension 1, this formula includes as a particular case (\(s=n-2\)) the  formula for the maximal models (see Theorem \ref{ulterioremiglioramentoformulamax}).
\end{rem}

Now we are ready to  describe formulas for the Poincar\'e polynomials of the models \(Y_{\G_{s}(A_{n-1})}\): these turn out to be interpolations between the well known formula for minimal models and the formula for maximal models of Theorem \ref{ulterioremiglioramentoformulamax}.
In these interpolations the  polynomials  \(P(Y_{{\widetilde \G}_{s}(A_{n-1})})(q)\)  play a role.

In  \cite{YuzBasi} 
the Poincar\'e series \(\Phi(q,t)=t+ \sum_{n\geq 2, i} dim \ H^{2i}(Y_{\G_{1}(A_{n-1})}, \Q) q^i \frac{t^n}{n!}\) for the minimal models has been computed in the following way. \footnote{Since the projective  minimal models are isomorphic to the moduli space \({\overline M}_{0,n+1}\), this series also appear in many papers, computed  from the moduli point of view: see for instance   \cite{getzler}, \cite{manin}.}
First  one computes, via a recursive relation,  the  series \(\lambda(q,t)\) which counts the contribution of basis monomials whose associated nested set is represented by a tree (included the degenerate tree given by a single leaf, which gives contribution \(t\)):
\begin{displaymath}
	\lambda(q,t) ^{(1)}=1 + \frac{\lambda(q,t) ^{(1)}}{q-1}
	\left[e^{q\lambda(q,t) }-qe^{\lambda(q,t) }+q-1    \right]
\end{displaymath}
(here the superscript \(^{(1)}\) means the first derivative with respect to \(t\)).

Then one obtains  \(\Phi(q,t)\) as  \(e^{\lambda(q,t) }-1\).

Now we need a modification  \(\Phi(q,t,y)\) of  \(\Phi(q,t)\),  where the powers of the variable \(y\)  take into account  the number of the maximal subspaces  in the nested sets associated to basis elements:
\[ \Phi(q,t,y)= e^{y\lambda(q,t) }-1\]

 \begin{teo}
 \label{formulafamigliainduttiva}
For every \(n\geq 3\) and  \(1\leq s\leq n-2\)  we have the following  formula for the Poincar\'e polynomial of the models \(Y_{\G_{s}(A_{n-1})}\):
 \[\frac{t^n}{n!}P(Y_{ \G_{s}(A_{n-1})})(q)=\frac{t^n}{n!}P(Y_{ {\widetilde G}_{s}(A_{n-1})})(q)+\]\[+\sum_{
\begin{array}{c}
s<j\leq n-2
\end{array}
}\Phi_{|{\small \begin{array}{c}
deg\ y=j\\
 deg\ t =n
 \end{array}}}(q,t,1)P(Y_{{\widetilde \G}_{s}(A_{j-1})})(q) \]
 \end{teo}
\begin{proof}
Our  first step consists in describing    the \(\G_{s}(A_{n-1})\)-nested sets, since they are the supports of the monomials of the Yuzvinski bases  (see  Section \ref{subsecwonderful}). 
One observes that, if 
\(S\) is a \(\G_{s}(A_{n-1})\)-nested set, 
then \(S\) can be partitioned into two subsets:\\
a) the (possibly empty) subset \(S_1\)  made by the subspaces  which  belong to \({\widetilde \G}_{s}(A_{n-1})\). If  \(S_1\)  is not empty,  it contains a minimal element \(A\)  (with respect to inclusion). Then the elements of \(S_1\) are totally orderd by inclusion (\(A\) is the minimal one).\\ 
b) the (possibly empty) subset \(S_2\) made by the subspaces  which  belong to \({\mathcal F}_{A_{n-1}}-{\widetilde \G}_{s}(A_{n-1})\). They  satisfy the following properties: 
they form a \({\mathcal F}_{A_{n-1}}\)-nested set;  their sum \(B\) doesn't belong to  \({\widetilde \G}_{s}(A_{n-1})\) and, 
if \(S_1\) is not empty and \(A\) is the minimal subspace in \(S_1\),   \(B\) is strictly included into \(A\).

Therefore, every monomial \(m\)  in the basis is the product of a monomial \(m_{S_1}\) with support in \(S_1\) and  a monomial  \(m_{S_2}\) with support in \(S_2\). 
We notice that \(m_{S_2}\) also belongs to the cohomology basis of \(Y_{ \F_{A_{n-1}}}\).

Let us denote by \(A_1,A_2,..., A_t\) the maximal elements in \(S_2\).   We can represent them by subsets of \(\{1,2,...,n\}\) as usual;  considering the cardinalities of these subsets, and  eventually adding some parts equal to 1,  we can associate to \(A_1,A_2,..., A_t\)  a partition \(\lambda \in \Lambda_n\). 

 One can then compute Poincar\'e polynomials by regrouping all the basis monomials such that the maximal elements  in \(S_2\) give a partition of length \(j\), with the two following restrictions on \(j\):  \(j\leq n- 2\)  (all the subspaces  \(A_1,A_2,..., A_t\) have dimension \(\geq 2\) otherwise they are not in the support of a basis element) and \(j>s\) (otherwise \(A_1+A_2+\cdots +  A_t\) belongs to  \(\G_{s}(A_{n-1})\)).

The contribution of all the ``\(m_{S_2}\) factors''  such that the maximal elements in \(S_2\)  give a partition of length \(j\) is provided (up to multiplication by \(\frac{t^n}{n!}\)) by: 
\[\Phi_{|{\small \begin{array}{c}
deg\ y=j\\
 deg\ t =n
 \end{array}}}(q,t,1)\]
Now we observe that, by our description of nested sets, once such a  factor \(m_{S_2}\) is fixed,  all the  factors of type \(m_{S_1}\) are in bijective correspondence with the monomials of  the  cohomology basis of \(Y_{{\widetilde \G}_{s}(A_{j-1})} \).  The following example illustrates this correspondence: let \(n=12\) and let \(\{1,2,3\} \) and \(\{4,5,6,7\}\) be the maximal subspaces  in \(S_2\). Then every subspace in \(S_1\) contains \(\{1,2,3\} \) and \(\{4,5,6,7\}\), therefore we can represent it as a partition of \(\{1,2,...,12\}\) where 1,2,3 belong to the same part, and 4,5,6,7 belong to the same part. Now, ``collapsing'' 1,2,3 to a new symbol \(\overline {1}\) and 4,5,6,7 to \(\overline {4}\), we are representing every subspace in \(S_1\) by  a partition of  \(\{\overline {1}, \overline {4},8,9,10,11,12\}\), or renumbering the elements, by  a partition of \(\{1,2,3,4,5,6,7\}\).  In this way we associate to the monomial  \(m_{S_1}\) a monomial in the  cohomology basis of \(Y_{{\widetilde \G}_{s}(A_{6})} \) (in this correspondence the   exponents do not change, according to the  definition of admissible monomials in Section \ref {subsecwonderful}).

 \end{proof}


\begin{es}
Here there are some  examples:
\[P(Y_{\G_{1}(A_{4})})(q)=q^3+16q^2+16q+1\]
\[P(Y_{\G_{2}(A_{4})})(q)=q^3+26q^2+26q+1\]
\[P(Y_{\G_{3}(A_{4})})(q)=q^3+41q^2+41q+1\]
\[P(Y_{\G_{1}(A_{5})})(q)=q^4+42q^3+127q^2+42q+1\]
\[P(Y_{\G_{2}(A_{5})})(q)=q^4+67q^3+222q^2+67q+1\]
\[P(Y_{\G_{3}(A_{5})})(q)=q^4+142q^3+372q^2+142q+1\]
\[P(Y_{\G_{4}(A_{5})})(q)=q^4+187q^3+732q^2+187q+1\]
\[ P(Y_{\G_{5}(A_{6})})(q)=q^5  + 855q^4  + 9556q^3  + 9556q^2  + 855q + 1\]
\end{es}

\begin{rem}[There is not an extended action on non minimal models]
As it is well known, the minimal model of type \(A_{n-1}\) has a natural `extended' \(S_{n+1}\) action, which comes from the moduli interpretation (see  for instance \cite{DCP1}, \cite{Gaiffitesi}, \cite{Gaiffidantonio}).
This is not true for the other \(S_n\) invariant models, as one can see from the geometrical point of view since the \(S_n\) action induced on the irreducible divisors in the boundary does not extend to a \(S_{n+1}\) action.  

From the algebraic point of view,  for instance, a comparison between the character \(\chi^2_4 \)  of the \(S_4\) action  on \(Y_{\G_{2}(A_{3})}\) and  the character \(\chi^1_4 \)  of the \(S_4\) action  on \(Y_{\G_{1}(A_{3})}\) shows that on the cohomology of the maximal model there is  not an  extended action compatible with the extended action on the cohomology of the minimal model. In fact  \(\chi^2_4-\chi^1_4=(s_{(4)}+s_{(2,2)})q\) and there is not a representation of \(S_5\) which, once restricted, decomposes as  \(s_{(4)}+s_{(2,2)}\).

\end{rem}

\section{Case \(B_n\) ( and \(C_n\)), classification of  all the invariant building sets}
\label{buildingbn}
Let us consider the root arrangement of type \(B_n\) in \(\C^n\) and let \(W(B_n)\) be  its Weyl group (the case \(C_n\) leads to the same arrangement). The subspaces in the building set of irreducibles $\F_{B_n}$  are of two  types,  strong subspaces and  weak subspaces. A  subspace of $(\C^n)^*$  is strong if its  annihilator  can be described by  the equation $x_{i_1}=\cdots
=x_{i_k}=0$. Then strong subspaces can  be put in bijective correspondence with the subsets of $\{1,\cdots,n\}$ of cardinality greater than
or equal to $1$ (we will call  such subsets strong when they represent a strong subspace). A subspace in $\F_{B_n}$ is weak  if its  annihilator can be described by   $\{x_{i_1}=\cdots
=x_{i_r}=-x_{j_1}=\cdots =-x_{j_s}\}$ (\(r+s\geq 2\)); therefore weak elements   are  in  bijective correspondence with the subsets of $\{1,\cdots,n\}$ of cardinality greater than or equal
to $2$ (such subsets will be called weak) equipped with a partition (possibly trivial)  into  two  parts. 

Therefore, the  subspaces in the maximal building set \({\mathcal C}_{B_n}\) can put in bijective correspondence with the partitions of $\{1,\cdots,n\}$ such that   each part is labelled  "weak" or "strong" and the  following  extra conditions are satisfied:  at most one part is strong, and if there is not a strong part, then at least one of the weak parts has more than 1 element.\footnote{In this notation we allow  the presence of  weak   singletons \(\{i\}\),  which are  associated to the  subspace \(\{0\}\).}

We want to classify all   the \(W(B_n)\) invariant building sets which include $\F_{B_n}$.  The combinatorial  description of the  subspaces in the maximal building set suggests us to introduce the notion of partition with a singular part:

\begin{defin} We denote by  \(S\Lambda_n\) the set of {\em singular  partitions}:   its  elements are  the couples \((r,\lambda)\)  with \(r\) integer, \(0\leq r \leq n\) and \(\lambda\in \Lambda_{n-r}\).  
\end{defin}
If  a subspace \(A\) in \({\mathcal C}_{B_n}\)  is represented by a partition of  $\{1,\cdots,n\}$, which has a strong part of cardinality \(r\) (\(r\) may  be 0) and weak parts whose cardinalities give the partition \(\lambda\in \Lambda_{n-r}\), we will say that  \(A\) has the form \((r,\lambda)\). 

The element   \((r,\lambda)\in S\Lambda_n\)  can be represented by a diagram whose higher row has \(r\) coloured boxes (see Figure \ref{tableaubn}).
  \begin{figure}[h]
 \center
\includegraphics[scale=0.25]{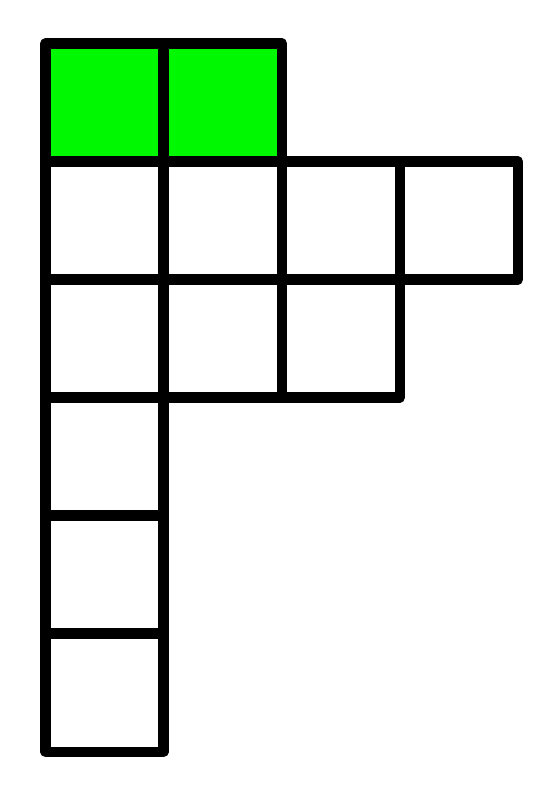}
\caption{The coloured diagram representing the singular (building) partition \((2,(4,3,1,1,1))\).}
\label{tableaubn}

\end{figure}

We  consider the following three  types of {\em admissible moves} on  \(S\Lambda_n\): \\
a) remove  an {\em entire} not coloured row and add all its boxes to a not coloured   row which has at least two  boxes or to   the coloured row (if it exists; if we are adding boxes to the coloured row, then the boxes will be coloured). At the end, if necessary,   the not coloured    rows will be rearranged in order to obtain a valid   diagram;\\
b)  remove \(k\geq 2\) not coloured rows made by a single box and form a row made by \(k\) boxes, if \(k\) is greater than or equal to the number of boxes of  the smallest  not coloured row with more than one box;  then, if necessary, rearrange the rows in order to obtain a valid coloured  diagram; \\
c) if the diagram is made by a single row, we can colour it.

As in the \(A_n\) case, we introduce a partial ordering in \(S\Lambda_n\): 
 \[(s,\gamma)\geq (r,\lambda) \]
 if and only if \((s,\gamma)\) can be obtained by \((r,\lambda)\) by a sequence of  admissible  moves.

\begin{defin}
A {\em singular building partition}  (of type \(B_n\), \(n\geq 2\)) is a couple \((r,\lambda)\in S\Lambda_n\)  which satisfies the further conditions that   \(\lambda\neq (1,1,1,...)\) and, if \(r=0\), then \(\lambda \in B\Lambda_n-\{(n)\}\).
We will denote by  \(SB\Lambda_n\)  the poset of all singular building partitions, with the ordering induced by \(\geq\).
\end{defin}


\begin{defin} Let us consider \((r,\lambda)\in SB\Lambda_n\).  We define the set \(\BG_{(r,\lambda)}\) as the union of  $\F_{B_n}$  with the set made by  all the subspaces of the form
 \((s,\gamma)\in SB\Lambda_n\), with   \((s,\gamma)\geq (r,\lambda)\).

\end{defin}

 We notice that, according to the definition, the building set of irreducibles is denoted by \(B\G_{(n,(0))}\). 
If \(n= 3\), there are two \(W(B_n)\) invariant  building sets which contain the irreducibles:  the minimal one \(B\G_{(3,(0))}\) and the maximal  one   \(B\G_{(1,(2))}\).  If \(n> 3\), the building sets \(B\G_{(s,\lambda)}\) (\((s,\lambda)\)  in  \(SB\Lambda_n\)) are all distinct and  the maximal building set  is described as \(B\G_{(0,(2,2,1,1,1,...))}\cup B\G_{(1,(2,1,1,1,1,...))}\).

\begin{prop}
Given \(n\geq 3\) and two different singular building partitions \((r,\mu)\) and  \((s,\lambda)\)  in  \(SB\Lambda_n\), the   building set \(B\G_{(s,\lambda)}\) is  included into \(B\G_{(r,\mu)}\) if and only if \((s,\lambda)>(r,\mu)\).
\end{prop}

 \begin{figure}[h]
 \center
\includegraphics[scale=0.25]{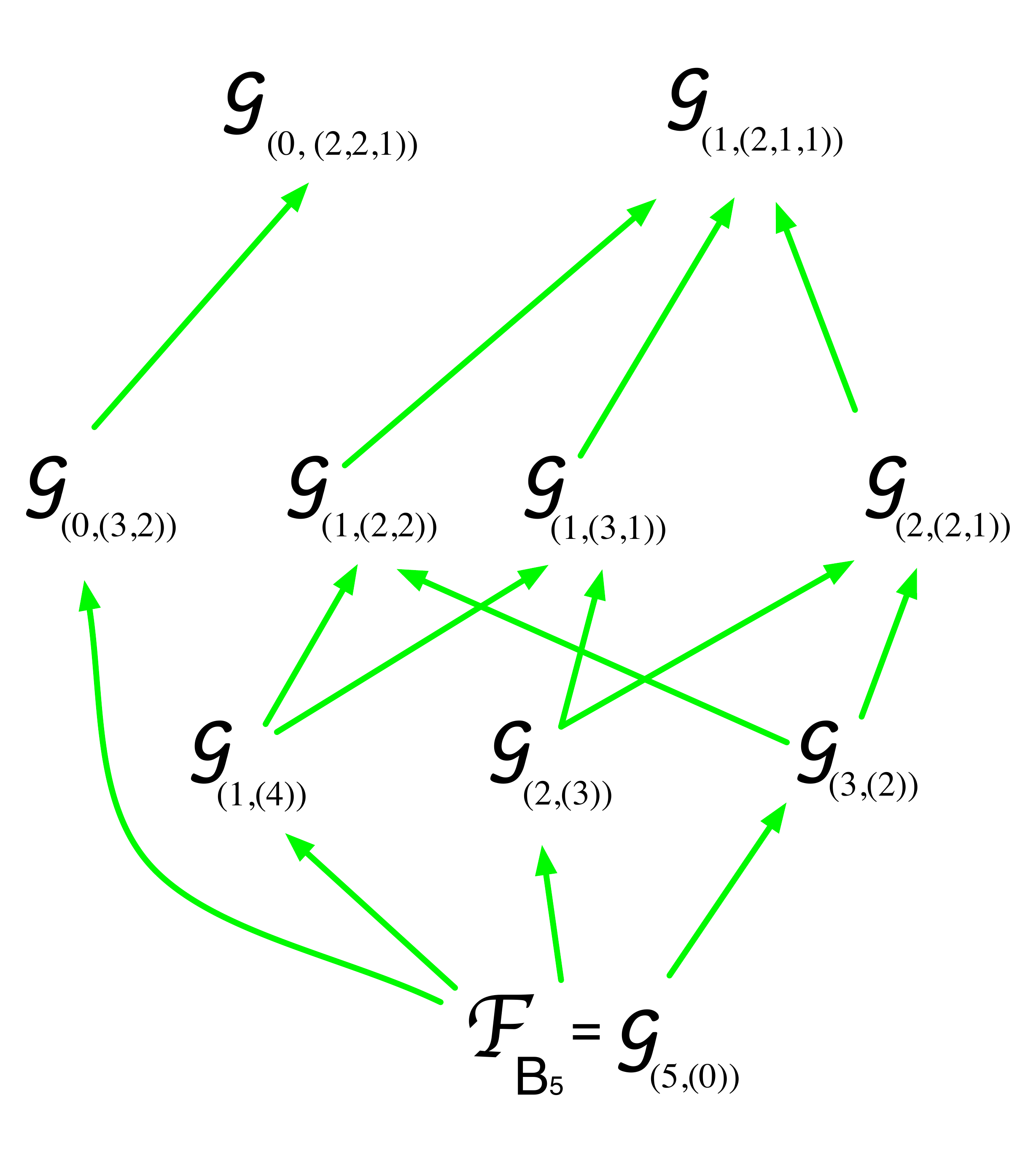}

\caption{The Hasse diagram of the family  of building sets of type \(\G_{(r,\mu)}\) (with \((r,\mu)\)  in  \(SB\Lambda_n\)) in the case \(B_5\).}
\label{famigliabuildingB5}
\end{figure}

\begin{defin}
Given a set \(\MS=\{(r_1,\lambda^1),(r_2,\lambda^2),...,(r_v,\lambda^v)\} \) of pairwise not comparable elements in \(SB\Lambda_n\), we denote by \(\G_{\MS}\) the building set 
\[\G_{\MS}=\G_{(r_1,\lambda^1)}\cup \G_{(r_2,\lambda^2)}\cup...\cup\G_{(r_v,\lambda^v)}\]
and by  \(B{\mathcal T}_n\) the set whose elements are the   nonempty sets of pairwise not comparable elements  in \(SB\Lambda_n\).
\end{defin}
We have the following classification theorem (we omit the proof, since it  is similar to the \(A_n\) case, Theorem \ref{teoclassificazione}).
 \begin{teo}
\label{teoclassificazionebn}
The map:
\[\MS=\{(r_1,\lambda^1),(r_2,\lambda^2),...,(r_v,\lambda^v)\} \rightarrow \G_{\MS}\]
is a bijection between \(B{\mathcal T}_n\)  and the set of the \(W(B_n)\) invariant building sets which contain the irreducibles.
\end{teo}

\section{Regular   building sets in case \(B_n\).  }
\label{secinductivefamiliesbn}

As in the  \(A_n\) case,  we focus on the {\em regular} \(W(B_n)\) invariant building sets,  obtained as the union of all the building sets of type \(\G_{(r,\mu)}\) (with \((r,\mu)\)  in  \(SB\Lambda_n\)) which lie on a same row of the Hasse diagram (see Figure \ref{famigliabuildingB5}).

\begin{defin}
For every \(n\geq 2\) and \(0\leq s\leq n-2\) we denote by \(\G_s(B_n)\) the (\(W(B_n)\) invariant) building set  which contains the irreducibles  and also all the subspaces of the maximal building set which have dimension \(\geq n-s\). We will call  \(\G_s(B_n)\) the {\em regular building set  of type \(B_n\) and of degree \(s\)}.
\end{defin}

For every \(n\geq 2\) and \(s\geq  n-1\) we put  \(\G_s(B_n)\) to be equal to the maximal building set.
We notice that, for \(n\geq 3\), \(\G_0(B_n)\) is the building set of irreducibles  (denoted by \(B\G_{(n,(0))}\) in Section \ref{buildingbn}) and, for \(n\geq 4\),  \(\G_{n-2}(B_n)\) is equal to the maximal building set \(B\G_{(0,(2,2,1,1,1,...))}\cup B\G_{(1,(2,1,1,1,1,...))}\). 

As in the \(A_n\) case, we will  define families  of subspace arrangements which are obtained by  removing some of  the irreducible subspaces from \(\G_s(B_n)\).
\begin{defin}
For every \(n\geq 2\) and \(0\leq s\leq n-1\) we denote by \({\widetilde \G}_s(B_n)\) the (\(W(B_n)\) invariant) building set  which contains  all the subspaces of the maximal building set which have dimension \(\geq n-s\). Moreover,  for every \(n\geq 2\) and \(s>  n-1\) we put  \({\widetilde \G}_s(B_n)\) to be equal to the maximal building set.
\end{defin}
We remark that \({\widetilde  \G}_{n-1}(B_n)= \G_{n-2}(B_n)\)  is  the maximal building set and that, for every fixed \(s\geq 0\),  the  family \({\widetilde \G}_s(B_n)\) is inductive.

%


Given two positive integers \(n>j\),  let us denote by \(h_{n,j}(q)\) the polynomial 
\[h_{n,j}(q)= \left( \sum_{k=1}^{n+1-(j-1)}\binom{n}{k-1}S(n+1-k,j-1)2^{n+1-(j-1)-k}\right)\frac{q^{n+1-j}-q}{q-1}\]
\begin{teo}
For every \(n\geq 2\) and \(0\leq s\leq n-2\) we have 
\label{formulabntilde}
\[P(Y_{{\widetilde \G}_{s}(B_n)})(q)=1+\sum_{
\begin{array}{c}
1\leq k \leq \lfloor \frac{s+2}{2} \rfloor\\
(j_1,j_2,...,j_k)\in J_k(s+1)
\end{array}}h_{j_2,j_1}(q)h_{j_3,j_2}(q)\cdots h_{j_k,j_{k-1}}(q)h_{n,j_k}(q)\]
where \(J_k(s+1)\) is the set of all the lists \((j_1,j_2,...,j_k)\) of integers such that \(1\leq j_1<j_2<\cdots <j_k\leq s+1\) and, for every \(i=1,2,...,k\), \(j_i-j_{i-1}\geq 2\).

\end{teo}
\begin{proof}
We can compute the Poincar\'e polynomials using  the   strategy  described in Section  \ref{secformulamax}, i.e.   by removing at each step a minimal element and considering the quotient.
  Since the family of  building sets \({\widetilde \G}_s(B_n)\) is inductive,   every quotient  is again a building set of type  \({\widetilde \G}_s(B_j)\).  We  then have the following inductive formula:
   \[P(Y_{{\widetilde \G}_{s}(B_n)})(q)=1+\sum_{
\begin{array}{c}
j=n-s+1,....,n 
\end{array}
}\frac{q^{j-1}-q}{q-1} \binom{n}{j} 2^{j-1}P(Y_{{\widetilde \G}_{s}(B_{n-j+1})})(q)+\]\[+\sum_{
\begin{array}{c}
j=n-s,....,n-1 
\end{array}
}\frac{q^{j}-q}{q-1} \binom{n}{j} P(Y_{{\widetilde \G}_{s}(B_{n-j})})(q)+  \]\[+\sum_{{\begin{array}{c}
(r,\lambda) \in SB\Lambda_n\\
 l(\lambda)\leq s
\end{array}}}\frac{q^{n-l(\lambda)}-q}{q-1} \binom{n}{r}2^{n-r-l(\lambda)}t_{\lambda} P(Y_{{\widetilde \G}_{s}(B_{l(\lambda)})})(q)\]
 The first (res. second) addendum describes the quotients by weak (resp. strong) subspaces in $\F_{B_n}\cap  {\widetilde \G}_{s}(B_n)\).
 The third addendum describes the quotients by subspaces in \({\widetilde \G}_s(B_n)\)  whose form \((r,\lambda)\) belongs to  \( SB\Lambda_n\).
 For every \(j=1,...,s+1\)  we can regroup all the subspaces which have dimension \(n+1-j\), which are 
 \[  \sum_{k=1}^{n+1-(j-1)}\binom{n}{k-1}S(n+1-k,j-1)2^{n+1-(j-1)-k}\]
 subspaces.
 Therefore we obtain 
  \[P(Y_{{\widetilde \G}_{s}(B_n)})(q)=1+\sum_{
\begin{array}{c}
1\leq j\leq s+1\end{array}}h_{n,j}(q) P(Y_{{\widetilde \G}_{s}(B_{j-1})})(q)\]

Since \(j\leq s+1\)  we have that   \({\widetilde \G}_{s}(B_{j-1})\)   is equal to the maximal building set   associated to the root arrangement \(B_{j-1}\) and the proof can be concluded by induction(as a base for the induction we put \(P(Y_{{\widetilde \G}_{s}(B_{0})})(q)=1\)).
\end{proof}

\begin{rem}
 Since \({\widetilde  \G}_{n-2}(B_n)\) and \(\G_{n-2}(B_n)\) differ only in the subspaces of dimension 1, this formula includes as a particular case (\(s=n-2\))  the formula  for maximal models (see \cite{GaiffiServenti}, where a formula was obtained using a different combinatorial argument).  
\end{rem}

Now we are ready to  describe formulas for the Poincar\'e polynomials of the models \(Y_{\G_{s}(B_n)}\):  as in the \(A_n\) case, these turn out to be interpolations between the formula for minimal models and the formula for maximal models.

In \cite{YuzBasi}, \cite{GaiffiBlowups}
the Poincar\'e series \[\Phi_B(q,t)=\sum_{n\geq 1, i}  dim \ H^{2i}(Y_{\G_{0}(B_n)}, \Q) q^i \frac{t^n}{2^nn!}\] for the minimal models has been computed in the following way.
Let \(\lambda_B(q,t)\) be the series  which counts the contribution of basis monomials whose associated nested set is represented by a tree which has only weak vertices. We have:  
\begin{displaymath}
	\lambda_B(q,t)=\frac{1}{2}\lambda(q,t) 
	\end{displaymath}
where \(\lambda(q,t)\) is the corresponding series for the \(A_{n}\) case (see Section \ref{secinductivefamilies}).

Then one observes that the series \(\mu_B\) which counts  the contribution of basis monomials whose associated nested set is represented by a tree which has at least a strong vertex is provided by the relation:
\[\mu_B(q,t)=\frac{1}{1-\gamma_B(q,t)}-1\]
where \[ \gamma_B(q,t)=\frac{e^{q\lambda_B(q,t)}-qe^{\lambda_B(q,t)}}{q-1}+1\]

Then one obtains  \(\Phi_B(q,t)\) as  \(e^{\lambda_B(q,t) }(\mu_B+1)-1\).

Now we need the following modification  \(\Phi_B(q,t,y)\) of  \(\Phi_B(q,t)\),  where the powers of the variable \(y\)  take into account the number of maximal subspaces in  the nested sets associated to basis elements:
\[ \Phi_B(q,t,y)= e^{y\lambda_B(q,t) }(y\mu_B+1)-1\]

 \begin{teo}
 \label{formulafamigliainduttivabn}
For every \(n\geq 2\) and  \(0\leq s\leq n-2\)  we have the following  formula for the Poincar\'e polynomial of the models \(Y_{\G_{s}(B_n)}\):
 \[\frac{t^n}{2^nn!}P(Y_{\G_{s}(B_n)})(q)=\frac{t^n}{2^nn!}P(Y_{{\widetilde \G}_{s}(B_n)})(q)+\]\[+\sum_{
\begin{array}{c}
s<j\leq n-2
\end{array}
}\Phi_B{_{|{\small \begin{array}{c}
deg\ y=j\\
 deg\ t =n
 \end{array}}}}(q,t,1)P(Y_{{\widetilde \G}_{s}(B_j)})(q) \]
 \end{teo}
\begin{proof}
The proof is similar to the one  in the $A_{n-1}$ case (see Theorem \ref{formulafamigliainduttiva}).
\end{proof}


\section{The interplay between boolean and root arrangements}
The boolean arrangement  is a subarrangement of the  arrangement of type \(B_n\) in \(\C^n\).  The irreducibles are the lines in \((\C^n)^*\)  whose  annihilators  are  the hyperplanes  \(x_i=0\). 
The maximal building set is given by the subspaces of   \((\C^n)^*\) whose  annihilators  are  the subspaces  \(x_{i_1}=x_{i_2}=\cdots = x_{i_k}=0\) (\(k=1,2,...,n\)).
 We can define regular models: 
\begin{defin}
Given \(n\geq 1\) and \(-1\leq s\leq n-2\) we denote by  \(\G_s(Bo(n))\) the  building set  which contains  the irreducible subspaces and also all the subspaces of the maximal building set which have dimension \(\geq n-s\). 

\end{defin} 

For \(n=1\) there is only one building set.  Given \(n\geq 2\), one immediately observes that the regular building sets  \(\G_s(Bo(n))\), with  \(-1\leq s\leq n-2\),  are all the \(S_n\) invariant building sets  which contain the irreducibles (for \(s=-1\) we have the building set of irreducibles, for \(s=n-2\) the maximal building set). For every fixed \(s\geq -1\),  the  family  \(\G_s(Bo(n))\) is inductive.

The maximal projective model \(Y_{ \G_{n-2}(Bo(n))}\) is isomorphic to the toric variety of type \(A_{n-1}\) (see Procesi \cite{procesitoric}, Henderson \cite{HenPisa}).
We observe that there is the following chain of inclusions among building sets:
\[\G_{-1}(Bo(n))\subsetneq \G_{0}(Bo(n))  \subsetneq \cdots   \subsetneq \G_{n-2}(Bo(n)) \subsetneq \G_0(B_n)\subsetneq \cdots \subsetneq \G_{n-2}(B_n)\]
Then we have  the following chain of projections among the associated models:
\[Y_{\G_{-1}(Bo(n))}\leftarrow Y_{\G_{0}(Bo(n))}  \leftarrow  \cdots  \leftarrow Y_{\G_{n-2}(Bo(n))} \leftarrow  Y_{ \G_0(B_n)}\leftarrow  \cdots \leftarrow Y_{\G_{n-2}(B_n)}\]
which gives ring injections in cohomology: 
 {\small \[H^*(Y_{\G_{-1}(Bo(n))})  \rightarrow H^*(Y_{\G_{0}(Bo(n))})  \rightarrow \cdots \rightarrow    H^*(Y_{\G_{n-2}(Bo(n))} ) \rightarrow   H^*(Y_{ \G_0(B_n)})\rightarrow    \cdots \rightarrow   H^*(Y_{\G_{n-2}(B_n)})\]}
 These ring injections can be described explicitely in terms of the bases, according to  the following general rule, which depends on the blow-up construction.
Let \(\T \subset \G\) be two building sets of subspaces of \(\C^n\).  
 For every \(A\in \T\) let us define \[\U_{\T,\G}(A)=\{B\in \G-\T \: | \: A \ \hbox{is maximal among the subspaces in} \ \T \ \hbox{which are included in} \ B\}\]

Then the  ring  injection
 \[R_{\T,\G}\: : \:  H^*(Y_{\T} ) \rightarrow  H^*(Y_{\G} ) \]
 is described by:
 \[\forall A\in \T \quad R_{\T,\G}(c_A)=c_A+\sum_{B\in \U_{\T,\G}(A)}c_B\]

For instance, if \(\T\) is equal to \({\mathcal F}_{A_{5}}\), \(\G=\G_{(2,2,1,1)}\) is the maximal model of the arrangement \(A_5\) and \(A=(1,2,3)\), we have:
\[R_{{\mathcal F}_{A_{5}},\G_{(2,2,1,1)}}(c_{(1,2,3)})=c_{(1,2,3)}+c_{(1,2,3)(4,5)}+c_{(1,2,3)(4,6)}+c_{(1,2,3)(5,6)}+c_{(1,2,3)(4,5,6)}\]
 
\section{Remarks on the  classification in case $D_n$}
\label{secdn}
Let us now consider the root arrangement of type $D_n$ in $\C^n$ and denote by $W(D_n)$ its Weyl group. The building set of
irreducibles is the same as in the  $B_n$ case,  except for the strong sets, which  must now  have cardinality at least 2. Hence, as in
the $B_n$ case, we can put the elements of the maximal model $\Cc_{D_n}$ in a bijective correspondence with the
partitions of $\{1,\ldots n\}$ with at most one strong part and  such that the strong part (if there is one) is required to have 
cardinality \(\geq2\).
\\
With this setting, in order to classify all the $W(D_n)$ invariant building sets which contain the irreducibles  we can repeat
the same arguments used  in the case $B_n$.  We start with  a slightly different set of couples, in fact we replace  $S\Lambda_n$ with $\overline{S}\Lambda_n$, where
$\overline{S}\Lambda_n$ is the set of couples $(r,\lambda)$ such that  $r\neq 1$ and $\lambda$ is a partition of
$n-r$. 
\begin{defin}
 A singular building partition of type $D_n$ ($n\geq 4$) is a couple $(r,\lambda)\in\overline{S}\Lambda$ with
$\lambda\neq (1,1,1,\ldots,1)$ and, if $r=0$, $\lambda\in B\Lambda_n -\{(n)\}$. We will denote by $SB\Lambda_n(D_n)$ the
poset of all singular building partitions with the ordering induced by $\geq$.
\end{defin}
\begin{defin} Let \(n\geq 4\). Let us consider
\((r,\lambda)\in SB\Lambda_n(D_n)\). If $n$ is odd or $r\neq 0$  we
define the set \(\DG_{(r,\lambda)}\) as the union of the building set of irreducibles of type \(D_n\)  with the set made
by  all the subspaces of the form $(s,\gamma)$ where $(s,\gamma)\in SB\Lambda_n(D_n)$ and $(s,\gamma)\geq (r,\lambda).$
\end{defin}
\begin{rem}
 If $n$ is an even number greater than or equal to $4$, $r=0$ and $\lambda$ is a partition of $n$ made by even numbers we
have two uncomparable $W(D_n)$ invariant building sets (containing the irreducibles of type $D_n$) associated to the
singular partition $(0,\lambda)$.  To see this suppose that $\lambda=(\lambda_1,\ldots,\lambda_m)$ and  call
\begin{itemize}
\item 
$I(\lambda,+)$ the set of subspaces whose form  is $\lambda$, and  in which every subset  \(\{i_1,i_2,...,i_k\}\) represents the 
annihilator of   $\{x_{i_1}=\cdots =x_{i_{k-1}}=x_{i_k}\}$;  
\item $I(\lambda,-)$ the set of subspaces whose form  is $\lambda$ in which the first subset  represents the annihilator of
 $\{-x_{i_1}=x_{i_2}=\cdots =x_{i_{\lambda_1-1}}=x_{i_{\lambda_1}}\}$ and the other subspaces are as in $I(\lambda,+)$. 

\end{itemize}
\end{rem}

As a consequence of this remark, when $n$ is even and $\lambda$ is a partition of $n$ made by even numbers, in the poset
of singular building partitions,  the vertex corresponding to $(0,\lambda)$ splits  into  two vertices $(0,\lambda,+)$
and $(0,\lambda,-)$; we have the same ``double vertex" in the corresponding poset of the building sets \(\DG_{(r,\lambda)}\): $D\G(0,\lambda,+)$ and
$D\G(0,\lambda,-)$.
\begin{prop}
 Given $n\geq 4$, and two different singular building partitions $(r,\lambda)$ and $(s,\gamma)$, we have that
$D\G_{(s,\gamma)}\subset D\G_{(r,\lambda)}$ if and only if $(s,\gamma)\geq (r,\lambda).$\footnote{The order relation \(\geq\) is the same as in the \(B_n\) case with the only difference that when we compare two partitions $(0,\lambda,\pm)$ and $(0,\gamma,\pm)$, in order to   have $(0,\lambda,\pm)\geq (0,\gamma,\pm)$ we also request that the signs coincide.}
\end{prop}

\begin{defin}
 Let $n$ be an odd number \( \geq 5\). 
 Given a set $\Sc=\{(r_1,\lambda^1),\ldots,(r_v,\lambda^v)\}$ of pairwise non comparable elements in
$SB\Lambda_n(D_n)$, we denote by $\G_\Sc$ the building set
$$\G_\Sc=D\G_{(r_1,\lambda^1)}\cup\ldots\cup D\G_{(r_v,\lambda^v)}.$$
\end{defin}
\begin{rem} If $n$ is even (\(n\geq 4\)) 
we have a similar definition  with respect to  to the poset with double vertices: 
 in the set of non
comparable elements $(0,\lambda,+)$ or $(0,\lambda,-)$ (or both) may appear.
\end{rem}

The proof of the following classification theorem is similar to the one in the cases $A_{n}$ and $B_n$: 
\begin{teo}
 Let $n\geq 4$. If $n$ is odd,  the $W(D_n)$ invariant building sets which contain the irreducibles are in bijection
with the  unions of sets of pairwise not comparable (with respect to inclusion) elements of the family $D\G_{(r,\lambda)}$
$((r,\lambda)\in SB\Lambda_n(D_n))$.\\
If $n$ is even we have the same statement, with respect to   the poset with double vertices. 
\end{teo}

\section{Regular building sets in case $D_n$}
\label{regdn}
We can
compute the Poincar\'e polynomial of the $D_n$ maximal model by subtracting from the $B_n$ one the contribution
provided by the basis monomials whose associated  $\Cc_{B_n}$-nested set  contains  at least an element  with  strong part of cardinality one. If we denote
by $\Gamma_\text{max}^n(q)$ such contribution we have the following
\begin{teo}
$$\Gamma_\text{max}^n(q)=n\left(\sum_{
\begin{array}{c}
 1\leq k \leq \lfloor \frac{n-1}{2} \rfloor \\
\,  \\
 (j_1,\ldots,j_{k+1})\in \tilde{J}_k(n-1)
\end{array}
}f(n-1,j_k)\cdots f(j_3,j_2)\tilde{f}(j_2,j_1)P_{\text{max,}B_{j_1}}(q)\right)$$
where, given, $n>m$
$$\tilde{f}(n,m):=S(n,m)\frac{q^{n-m+1}-q}{q-1}$$
and $\tilde{J}_k(n-1)$ is the set of $(k+1)$-tuples $(j_1,\ldots,j_{k+1})$ such that $1\leq j_1<j_2<\ldots<j_k<j_{k+1}=n-1$,
$j_2-j_1\geq 1$ and $j_i-j_{i-1}\geq 2$ for $i\in\{3,\ldots,k+1\}$.
\end{teo}
\proof
Fix $n\geq 2$ and let us see how a "bad" nested set is  done. Since the nested sets of the maximal model are totally
ordered by inclusion we must have a minimum element, say $N_1$, with strong part of
cardinality one. Hence, $N_1$ must have a weak part given by a partition in $j\geq 1$ (and \(j<n-1\)) parts of the remaining $n-1$
leaves
(otherwise the corresponding monomials wouldn't be admissible). The subspaces  included in this one, in the nested set we are dealing with, have no
strong part (by minimality) and they are obtained by splitting (in an admissible way) the weak part of $N_1$; on the
other hand, the family of the subspaces which lie  above $N_1$ may be thought as an admissibile $\Cc_{B_j}$ nested set.\\
From these remarks and since the strong leaf may be chosen in $n$ different ways, the claim  follows.\qed
\\ 
\begin{cor}
 $$P_{\text{max,}D_n}=P_{\text{max,}B_n}-\Gamma_\text{max}^n(q).$$
\end{cor}
\proof
Immediate from the theorem and the remarks above.\qed
\\
\begin{defin}
 For every $n \geq 4$ and $0 \leq s \leq n − 2$ we denote by $\G_s(D_n)$ the
($W(D_n)$ invariant) building set which contains the irreducibles and also all the
subspaces of the maximal building set which have dimension $\geq n- s$. We will call $\G_s(D_n)$ the {\em regular building sets of type \(D_n\) and degree \(s\)}. Moreover,
for every $n\geq 4$ and $s \geq n − 1$ we put $\G_s (D_n )$ to be equal to the maximal building
set.
\end{defin}

\begin{defin}
For every $n \geq 4$ and $0 \leq s \leq n − 2$ we denote by $\tilde{\G}_s (D_n)$ the
$W(D_n)$ invariant building set which contains all the subspaces of the maximal
building set which have dimension $\geq n -s$.  Moreover, for every $n \geq 4$ and
$s \geq n − 1$ we put $\tilde{\G}_s (D_n )$ to be equal to the maximal building set.
\end{defin}

Now,as in the case of the maximal model, we can compute the Poincar\'e polynomial of the models $Y_{\tilde{\G}_s(D_n)}$
starting from the ones of the models $Y_{\tilde{\G}_s(B_n)}$.
\begin{teo}
 For every $n\geq 4$ and $0\leq s \leq n-2$
$$P(Y_{\tilde{\G}_s(D_n)})(q)=P(Y_{\tilde{\G}_s(B_n)})(q)-\Gamma_s^n(q)$$ where
$$\Gamma_s^n(q)=n\left(\sum_{
\begin{array}{c}
 1\leq k \leq \lfloor \frac{s+2}{2} \rfloor \\
 \\
 (j_1,\ldots,j_{k+1})\in \tilde{J}_{k,s}(n-1)
\end{array}
}f(n-1,j_k)\cdots f(j_3,j_2)\tilde{f}(j_2,j_1)P(Y_{\tilde{\G}_s(B_{j_1})})(q)\right)
$$
where  \(\tilde{J}_{k,s}(n-1) \)is the set of $(k+1)$-tuples $(j_1,\ldots,j_{k+1})$ such that $1\leq j_1<j_2<\ldots<j_k<j_{k+1}=n-1$,
\(j_k\leq s\), $j_2-j_1\geq 1$ and $j_i-j_{i-1}\geq 2$ for $i\in\{3,\ldots,k+1\}$.
\end{teo}
\proof
The proof is essentially the same as in the maximal case: we have to subtract to $P(Y_{\tilde{\G}_s(B_n)})(q)$ the
contribution provided by the monomials whose associated nested sets contain an element with a strong part of
cardinality one. The only difference is that now we are dealing with subspaces of dimension at least $n-s$.\\
\qed

\begin{rem}
 If $n<4$, as bases for the induction, we take   $P(Y_{\widetilde{\G}_0(D_1)})(q)=1$;
$P(Y_{\widetilde{\G}_0(D_2)})(q)=P(Y_{\widetilde{\G}_1(D_2)})(q)=1+q$;
$P(Y_{\widetilde{\G}_0(D_3)})(q)=1+q+q^2$;$P(Y_{\widetilde{\G}_1(D_3)})(q)=P(Y_{\widetilde{\G}_2(D_3)})(q)=1+7q+q^2$.\\
\end{rem}

The same strategy  (start from what with know about $B_n$ and subtract) may be applied to 
the computation of $P(Y_{\G_s(D_n)})(q)$.\\
The main difference  is that the strong irreducible sets in case  $D_n$ must have cardinality at least three while $B_n$
admits strong irreducible sets of cardinality two. So we may work as follows: if we define
$$\gamma_D(q,t):=2\left(\gamma_B(q,t)-q\frac{t^2}{2!2^2}-(4q+q^2)\frac{t^3}{3!2^3}-q\sum_{n\geq
4}\binom{n}{2}\frac{t^n}{n!2^n}\right)$$
then the contribution of the strong trees (in the Poincar\'e series of the minimal model) is given by
$$\mu_D(q,t)=\frac{1}{1-\gamma_D(q,t)}-1.$$
Calling $\Phi_D(q,t,y):=e^{y\lambda_A(q,t)}(y\mu_D(q,t)+1)-1$ we have:
\begin{teo}
 For every \(n\geq 4\) and  \(0\leq s\leq n-2\)  we have the following  formula for the Poincar\'e polynomial of the
models \(Y_{\G_{s}(D_n)}\):
 \[\frac{t^n}{2^{n-1}n!}P(Y_{\G_{s}(D_n)})(q)=\frac{t^n}{2^{n-1}n!}P(Y_{{\widetilde \G}_{s}(D_n)})(q)+\]\[+\sum_{
\begin{array}{c}
s<j\leq n-2
\end{array}
}\Phi_{D_{|{\small \begin{array}{c}
deg\ y=j\\
 deg\ t =n
 \end{array}}}}(q,t,1)P(Y_{{\widetilde \G}_{s}(D_j)})(q) \]
\end{teo}
\begin{proof}
The proof is similar to the one  in the $A_{n-1}$ case (see Theorem \ref{formulafamigliainduttiva}).
\end{proof}

\section{The Euler characteristic of  real models}
\label{seceuler}
The De Concini-Procesi construction can be repeated also for real  subspace arrangements and its projective version produces real
compact models.  The cohomology of these models has been described by Rains in \cite{rains}. In this section we will make a remark about Euler characteristic.

Let us consider  a real building set of subspaces in an euclidean vector space and denote by \(Y_{\G}(\C) \) and \(Y_{\G}(\R) \)
the complex model and the real compact model associated to it.   From a result of  \cite{krasnov} it follows   that
\(H^{2i}(Y_{\G}(\C), \Z_2)\cong H^{i}(Y_{\G}(\R), \Z_2)\);  therefore \(\sum_{i}(-1)^i \dim\ H^{2i}(Y_{\G}(\C), \Q)\)
 is equal to the Euler characteristic  \(\chi_E (Y_{\G}(\R))\). 
Then  if we put \(q=-1\) in our formulas for the Poincar\'e polynomials we obtain   the Euler characteristic of the corresponding
 real compact  De Concini-Procesi models.

We point out  that  there are other ways to compute the Euler characteristic of these models. For instance, as it is well known, in the \(A_{n-1}\) case the maximal real  compact model  can be obtained by gluing  \(n!\)   permutohedra of dimension
\(n-2\). Therefore another   formula for the Euler characteristic can  be obtained by counting the faces of the \(n!\) 
permutohedra and taking into account their identifications (a face of dimension \(i\)  is identified with
\(2^{n-i-1}-1\) other \(i\)-dimensional faces). More precisely, let \(P_{n-2}\) be the \((n-2)\)-dimensional permutohedron.
Then the  Euler characteristic of the {\bf real maximal} model \(Y_{max,n}(\R)\) is provided by the following formula:

\[\chi_E(Y_{max,n}(\R))= \sum_{i=0}^{n-2} \, (-1)^i\frac{  | i-\hbox{dimensional faces of}\;  P_{n-2}|}{2^{n-i-1}}\, n!=\]
\begin{equation} 
\label{eulero1}=\sum_{i=0}^{n-2} \, (-1)^i\frac{ S(n-1, n-1-i)}{2^{n-i-1}}\,(n-i-1)! n!
\end{equation}


From the  formula of Theorem \ref{ulterioremiglioramentoformulamax}, since \(f_{n,j}(-1)\) is equal to 0 if \(n,j\) have different parity and is equal to \(-S(n,j)\) otherwise, we obtain:

\[ \chi_E(Y_{max,n}(\R))= P(Y_{max,n})(-1)=1+
\]
\[+\sum_{\begin{array}{c}
 1\leq k\leq \lfloor \frac{n-1}{2} \rfloor\\
1\leq j_1<j_2<\cdots <j_k\leq n-2\\
n\equiv j_k \equiv \cdots \equiv j_1 \mod 2
\end{array}}(-1)^k S(j_2,j_1)S(j_3,j_2)\cdots S(j_k,j_{k-1})S(n,j_k)
\]

When \(n\) is odd this sum is easily shown to be equal to 0 (this is in accordance with Poincar\'e duality), while for \(n\) even the formula above specializes to:
\begin{equation} 
\label{eulero2} \chi_E(Y_{max,n}(\R))= P(Y_{max,n})(-1)=1+\end{equation}
\[+\sum_{\begin{array}{c}
 1\leq k\leq \lfloor \frac{n-1}{2} \rfloor\\
1\leq j_1<j_2<\cdots <j_k\leq n-2\\
j_i \ even
\end{array}}(-1)^k S(j_2,j_1)S(j_3,j_2)\cdots S(j_k,j_{k-1})S(n,j_k)
\]

For instance, when \(n=6\)
\[\chi_E(Y_{max,n}(\R))= 1-S(6,4)-S(6,2)+S(6,4)S(4,2)=1-65-31+455=360\]
We point out that by comparing formulas (\ref{eulero1}) and (\ref{eulero2}) some nice relations, involving Stirling numbers of the second kind, appear.
This remark extends to all the De Concini-Procesi models of  root arrangements, which are obtained by gluing nestohedra (see \cite{zelevinski}),  in particular to all the regular models.  For instance, the maximal model in case  \(B_n\) is obtained by gluing \(2^n n!\) permutohedra \(P(n-1)\), therefore by computing in two different ways the Euler characteristic one obtains that 
\[\sum_{i=0}^{n-1} \, (-1)^i\frac{ S(n, n-i)}{2^{n-i}}\,(n-i)! 2^nn!  \]
is equal to the number obtained  putting \(s=n-2\) and \(q=-1\) in the formula of Theorem \ref{formulabntilde}.
We remark that in in \cite{GaiffiServenti} one can find other different formulas for the Euler characteristic of the maximal models of root arrangements.

\addcontentsline{toc}{section}{References}
\bibliographystyle{acm}
\bibliography{Bibliogpre} 
\end{document}